\documentclass[11pt,leqno,english,draft]{smfart}
\usepackage{anysize}
\marginsize{3cm}{3cm}{3cm}{4.3cm}
\usepackage{amsmath}
\usepackage{amsfonts,amssymb}
\usepackage{enumerate}
\usepackage{amsthm}

\theoremstyle{plain}
\newtheorem{theo}{Theorem}[section]
\newtheorem{prop}[theo]{Proposition}
\newtheorem{lemm}[theo]{Lemma}
\newtheorem{coro}[theo]{Corollary}
\newtheorem{defi}[theo]{Definition}
\theoremstyle{definition}
\newtheorem{rema}[theo]{Remark}

\DeclareMathOperator{\cnx}{div}
\DeclareMathOperator{\RE}{Re}
\DeclareMathOperator{\dist}{dist}
\DeclareMathOperator{\IM}{Im}

\DeclareMathOperator{\supp}{supp}

\def\defn{\mathrel{:=}}
\def\delt{\varrho}

\def\Deltayx{\Delta_{x,y}}

\def\eps{\varepsilon}

\def\la{\left\lvert}
\def\lA{\left\lVert}
\def\le{\leq}
\def\L#1{\langle{#1}\rangle}

\def\Ly{(h^\delta\partial_y)}

\def\mez{\frac{1}{2}}

\def\partialx{\nabla}
\def\partialyx{\nabla_{x,y}}

\def\ra{\right\rvert}
\def\rA{\right\rVert}
\def\s{\sigma}

\def\tdm{\frac{3}{2}}
\def\tq{\frac{3}{4}}

\def\xC{\mathbf{C}}
\def\xN{\mathbf{N}}
\def\xR{\mathbf{R}}
\def\xT{\mathbf{T}}

\numberwithin{equation}{section}

\title{Strichartz estimates for water waves}

\author{T. Alazard}\address{ CNRS \& Univ Paris-Sud 11\\ D\'epartement de Math\'ematiques\\ F-91405 Orsay}
\email{thomas.alazard@math.u-psud.fr}
\author{N. Burq}\address{Universit\'e Paris-Sud 11 \\ D\'epartement de Math\'ematiques; CNRS \\ F-91405 Orsay}
\email{nicolas.burq@math.u-psud.fr}
\author{C. Zuily}\address{ Universit\'e Paris-Sud 11 \\ D\'epartement de Math\'ematiques; CNRS  \\ F-91405 Orsay}
\email{claude.zuily@math.u-psud.fr}

\thanks{Support by the French Agence Nationale de la Recherche, project EDP Dispersives, r\'ef\'erence ANR-07-BLAN-0250, is acknowledged.}

\date{\today}
\begin{abstract}In this paper we investigate the dispersive properties of the solutions of the two dimensional water-waves system. First we prove Strichartz type estimates with loss of 
derivatives at the same low level of regularity we were able to construct the solutions 
in [2]. On the other hand, for smoother initial data, we prove that the solutions enjoy 
the optimal Strichartz estimates (i.e, without loss of regularity compared to the system 
linearized at ($\eta =0, \psi = 0$)). 
\end{abstract}
 
\begin{document}
\maketitle
 
\section{Introduction}



In a time-dependent  domain $\Omega_t \subset \xR^{d+1}$ which 
is located between a free hypersurface $\Sigma_t$ and a fixed known bottom $\Gamma$, consider a potential flow  $v=\nabla_{x,y}\phi$, with
\begin{equation*}
\Deltayx\phi=0 \quad\text{in }\Omega_t,\quad \partial_{n}\phi=0  \quad\text{on }\Gamma.
\end{equation*}
The surface-tension water-waves problem is given by two equations: 
a kinematic condition (which states that 
the free surface moves with the fluid), and a dynamic condition (that expresses a 
balance of forces across the free surface). 
The system reads 
\begin{equation}\label{WW}
\left\{
\begin{aligned}
&\partial_{t} \eta = \partial_{y}\phi -\partialx\eta\cdot\partialx \phi &&\text{on }\Sigma_t=\{y=\eta(t,x)\}, \\
&\partial_{t}\phi+\frac{1}{2}\la \partialyx\phi\ra^2  +g \eta =  H(\eta)
&&\text{on }\Sigma_t,
\end{aligned}
\right.
\end{equation}
where $\nabla = \nabla_x$, $g>0$ is the acceleration of gravity and 
$$
H(\eta)= \text{div}\left( \frac{\nabla\eta}{\sqrt{1+(\partial_x \eta)^2}}\right).
$$
is the mean curvature of the free surface.

\subsection{Assumptions}
We work in a fluid domain such that there is uniformly a minimum depth of water, more precisely we assume that for each time $t$ one has
$$
\Omega_t=\Omega_{1,t}\cap \Omega_2
$$
where $\Omega_{1,t}$ is the half space located below the free surface $\Sigma_t$,
$$
\Omega_{1,t} = \{\,(x,y)\in \xR^d\times\xR  \,:\,   y<\eta (t,x)\,\} \qquad (d\ge 1)
$$
for some unknown function $\eta$ and $\Omega_{2}$ contains a fixed strip around $\Sigma_t$, that means that there exists $h>0$ such that, 
\begin{equation}\label{Ht} 
\{ (x,y)\in \xR^d\times \xR\, : \, \eta(t,x) -h \le y \leq \eta(t,x) \} \subset \Omega_2,
\end{equation}
for all $t\in [0,T]$. 
We shall also assume that the domain $\Omega_2$ (and hence the domain $\Omega_t=\Omega_{1,t}\cap \Omega_2$) 
is connected.

We emphasize that no regularity assumption is made on the bottom $\Gamma=\partial\Omega_t\setminus \Sigma_t$. 
We consider both cases of infinite depth and bounded depth bottoms (and all cases in-between). 
Finally, we could consider the cases where the free surface is a graph over a given smooth hypersurface and the bottom is time dependent.

\subsection{Main results}

 Following Zakharov we reduce the system to a system on the free surface. 
If $\psi=\psi(t,x) \in\xR$ is defined by 
$$
\psi(t,x)=\phi(t,x,\eta(t,x)),
$$
then $\phi(t,x,y)$ is the unique variational solution of 
\begin{equation}\label{dphi}
\Delta \phi = 0 \quad\text{in } \Omega_t, \qquad \phi(t,x, \eta (t,x) ) = \psi (t,x).
\end{equation}
The Dirichlet-Neumann operator is then defined by 
\begin{align*}
(G(\eta) \psi)  (t,x)&=
\sqrt{1+|\partialx\eta|^2}\,
\partial _n \phi\arrowvert_{y=\eta(t,x)}=\partial_y \phi-\partialx \eta \cdot \partialx \phi \Big\arrowvert _{y=\eta(t,x)}.
\end{align*}
(we refer to Section~2 in \cite {ABZ}
for a precise construction). 

Then $(\eta,\phi)$ is solution of the water-waves system~\eqref{WW} if and only if $(\eta,\psi)$ solves the system
\begin{equation}\label{system}
\left\{
\begin{aligned}
&\partial_{t}\eta-G(\eta)\psi=0,\\
&\partial_{t}\psi+g \eta- H(\eta)
+ \frac{1}{2}\la\partial_x \psi\ra^2  -\frac{1}{2}
\frac{\bigl(\partial_x  \eta\cdot\partial_x \psi +G(\eta) \psi \bigr)^2}{1+|\partial_x  \eta|^2}
= 0.
\end{aligned}
\right.
\end{equation}
Concerning the Cauchy theory, there are many results starting from the pionneering work of K.~Beyer and M.~G{\"u}nther \cite{BG}. See 
S.Wu~\cite{Wu0}, D.~M. Ambrose and N.~Masmoudi \cite{AmMa}, B.~Schweiser \cite{Sch}, 
T.~Iguchi \cite{Iguchi}, D.~Coutand and S.~Shkoller \cite{CS}, J. Shatah and C. Zeng \cite{SZ}, 
M.~Ming and Z.~Zhang \cite{MZ}, F. Rousset and N. Tzvetkov \cite{RT}. 
In~\cite{ABZ}, we established new local well posedness results for the system~\eqref{system} under sharp (as long as no dispersive effects are taken into account) regularity assumptions on the initial data. We refer 
to the introduction of~\cite{ABZ} for references and a short historical survey of the 
background of these problems. 

The purpose of this work is precisely, in the case $d = 1$, to investigate the dispersive properties of these solutions.  Our results are twofold: first we prove  Strichartz type estimates with loss of derivatives at the very same level of regularity we were able to construct the solutions in~\cite{ABZ}. On the other hand, for smoother initial data,  we  prove that the solutions enjoy the optimal Strichartz estimates (i.e, without loss of regularity compared to the system linearized at $(\eta=0, \psi=0)$). More precisely, our main results are the following.
\begin{theo}\label{theo:main}
Let $s>5/2$ and $T>0$. Consider a   solution 
$(\eta,\psi)$ of ~\eqref{system} on the time interval $I = [0,T]$ 
such that $\Omega_t$ satisfies~\eqref{Ht}   for $t\in I$. If 
$$
(\eta,\psi)\in C^0\big(I, H^{s+\mez}(\xR)\times H^{s}(\xR)\big),
$$
then 
$$
(\eta,\psi)\in L^4\big(I, W^{s+\frac{1}{4},\infty}(\xR)\times W^{s-\frac{1}{4},\infty}(\xR)\big).
$$
\end{theo}

\begin{theo}\label{theo:classic}
Let $s>11/2$, $T>0$ and $p,q,\s$ be such that
$$
\frac 2 p + \frac 1 q =\frac 1 2 ,\quad 2\le q < +\infty.
$$
Consider a   solution 
$(\eta,\psi)$ of ~\eqref{system} on the time interval $I = [0,T]$ 
such that $\Omega_t$ satisfies~\eqref{Ht}   for $t\in I$. If 
$$
(\eta,\psi)\in C^0\big(I, H^{s+\mez}(\xR)\times H^{s}(\xR)\big),
$$
then 
$$
(\eta,\psi)\in L^p\big(I, W^{s+\frac 3 8 +\frac 1 {4q},q}(\xR)\times W^{s- \frac 1 8 + \frac 1{4q},q}(\xR)\big).
$$
\end{theo}

\begin{rema}\ 
$(i)$ Theorem~\ref{theo:main} was obtained recently under the assumption $s\geq 15$ by Christianson-Hur-Staffilani~\cite{CHS} .

$(ii)$ Let $s>5/2$ and $(\eta_0, \psi_0) \in H^{s+ \mez}( \xR) \times H^s( \xR)$  satisfying $\dist (\Sigma_0,\Gamma)\ge c>0$, we proved in~\cite{ABZ} that there exist $T>0$ and a solution  $(\eta,\psi) \in C^0\big([0,T];H^{s+\mez}(\xR)\times H^{s}(\xR)\big)$ satisfying $\dist (\Sigma_t,\Gamma)\ge c>0$.

$(iii)$ Letting $q$ tend to infinity  we see that the result in Theorem~\ref{theo:classic} exhibits a gain of $1/8$ derivatives with respect to Theorem~\ref{theo:main}

$(iv)$ For the end point $(p,q)= (4,+\infty)$ we prove in fact, under the assumptions in Theorem~\ref{theo:classic}, that
$$
(\eta,\psi)\in L^4\big(I,B^{s+\frac 3 8}_{\infty,2}(\xR)\times B^{s- \frac 1 8}_{\infty,2}(\xR)\big).
$$
where $B^\s_{\infty,2}$ is the standard Besov space (see Section~\ref{section.6})

$(v)$ The gain of regularity exhibited in Theorem~\ref{theo:classic} is optimal as can be seen at the level of the linearized system around the trivial solution $(\eta,\psi)=(0,0)$ which reads (when $g=0$), 
$$
\partial_t\eta -\la D_x\ra \psi=0, \quad \partial_t \psi  -\Delta \eta=0.
$$
Indeed $u=\la D_x\ra^{\mez}\eta +i \psi$ is a solution of the equation 
$i\partial_t u- \la D_x\ra^\tdm u =0$, for which one can prove the optimal estimate
$$
\lA \exp(-it \la D_x\ra^\tdm) u_0\rA_{ L^4\left(I,W^{s-\frac{1}{8},\infty}(\xR)\right)}\leq C \|u_0\|_{H^{s}},
$$
which gives the desired regularity on $(\eta,\psi)$.

$(vi)$ It is most likely that Theorem~\ref{theo:main} remains valid when $\xR$ is replaced by the one dimensional torus $\xT$. Indeed, our proof relies on a semi-classical parametrix (on time intervals taylored to the frequency) which exhibits finite speed of propagation and which can consequently be easily localized in space.

$(vii)$ Notice that the dispersive estimates proved in this paper can be combined with our previous work to improve the regularity threshold obtained in~\cite{ABZ} and give local well posednesss for initial data below the $s= 2 + \frac 1 2 $ threshold. This will be the matter of a forthcoming paper (including the $3$-d water-waves system)~\cite{ABZ3}.

$(viii)$ Notice finally that dispersive properties of the operator linearized at $(\eta=0, \psi=0)$ were used recently by Wu~\cite{Wu1, Wu2} and Germain-Masmoudi-Shatah~\cite{GMS} to prove global existence results.
\end{rema}

\subsection{Strategy of the proofs}
Following the approach in Alazard-M\'etivier~\cite{AM}, after suitable paralinearizations, we have shown in 
~\cite{ABZ} that the water waves system can be arranged into an explicit paradifferential symmetric equation of Schr\"odinger type, and we deduced  the smoothing effect for the  2-d surface tension water waves. Here, we will also take benefit of this paralinearization reduction, and this reduced system will be our starting point.  
The guiding line for the rest of our proof is  very classical: construction of a parametrix to prove dispersion ($L^1- L^\infty$ estimates), and then $TT^*$ argument. 

There are two  main difficulties in the analysis of this equation. First the coefficients of the operator are time dependent and consequently we cannot get rid of the lower order terms by simple conjugation arguments (see  Burq-Planchon \cite{BuPl}).  Second the coefficients enjoy poor regularity, and  finally, whereas the principal part in the operator is of order $3/2$, the subprincipal part in the operator is of order $1$ which gives only a $1/2$ difference compared to the usual $1$ difference  encountered for magnetic Schr\"odinger operators. As will be shown in our analysis, the presence of such subprincipal parts will produce non trivial oscillations which here have to be taken into account in the analysis. 

The first common step for  both theorems is to perform several reductions for the paradifferential equation. The first one  is  to use Alinhac's para-composition theory~\cite{Alipara} (see also Burq-Planchon~\cite{BuPl} where a similar idea was used) to reduce the matters to the study of a Schr\"odinger type operator with constant coefficients principal part.  This is  particular to space dimension $1$ and reflects the fact that there is only one metric on $\xR$. The second reduction, inspired by works by Smith~\cite{Sm} and Bahouri-Chemin~\cite{BaCh}, consists in smoothing out the coefficients of the operator.

Once this reduction has been achieved, we can construct the parametrix, for which the natural time is the semi-classical one: $s= t|\xi|^{-1/2}$. Here the differences between our two theorems appear. Indeed, in the proof of Theorem~\ref{theo:main}, following the strategy in Burq-G\'erard-Tzvetkov~\cite{BGT1}, we construct the parametrix on small times $|s| \leq c$) and the main difficulty is to handle sharp regularity threshold (for smooth enough initial data the proof would be much simpler). In the proof of Theorem~\ref{theo:classic} the difficulties are different: first we have to handle the  oscillations generated by the subprincipal part and furthermore  we have to prove very large time asymptotics ($|s| \leq c |\xi|^{1/2} $) in the high frequency regime $|\xi| \rightarrow + \infty$. Notice that, even for  initial  data with arbitrarily large smoothness, the analysis would be non trivial. Finally, once the parametrix is constructed, the dispersion estimate is obtained by using  non classical stationary phase lemmas  involving precise controls on the remainder terms.

\section{Preliminaries}
In this section we recall some notations and results from \cite{ABZ} which will be used in the sequel. 
\subsection{Paradifferential calculus}In this paragraph we 
review  classical facts about Bony's paradifferential calculus. 

For $\rho\in\xN$, according to the usual definition, we denote 
by $W^{\rho,\infty}(\xR)$ the Sobolev spaces of $L^\infty$ functions 
whose derivatives of order $\le\rho$ are in $L^\infty$. 
For $\rho\in ]0,+\infty[\setminus \xN$, we denote 
by $W^{\rho,\infty}(\xR)$ the 
space of bounded functions whose derivatives of order $[\rho]$  are uniformly H\"older continuous with 
exponent $\rho- [\rho]$. 

\begin{defi}
Given $\rho\ge 0$ and $m\in\xR$, $\Gamma_{\rho}^{m}(\xR)$ denotes the space of
functions $a(x,\xi)$ on $\xR\times(\xR\setminus 0)$,
which are $C^\infty$ with respect to $\xi$ and
such that, for all $\alpha\in\xN $ and all $\xi\neq 0$, the function
$x\mapsto \partial_\xi^\alpha a(x,\xi)$ belongs to $W^{\rho,\infty}(\xR)$ and there exists a constant
$C_\alpha$ such that,
\begin{equation}\label{para:10}
\forall\la \xi\ra\ge \mez,\quad \lA \partial_\xi^\alpha a(\cdot,\xi)\rA_{W^{\rho,\infty}(\xR)}\le C_\alpha
(1+\la\xi\ra)^{m-\la\alpha\ra}.
\end{equation}
\end{defi}

\begin{defi}
$\Sigma_{\rho}^{m}(\xR)$ denotes the space of 
symbols 
$a(x,\xi)$ such that 
$$
a=\sum_{0\le j <\rho}a^{(m-j)} \qquad (j\in \xN),
$$
where $a^{(m-j)}\in \Gamma^{m-j}_{\rho-j}(\xR)$ is homogeneous of degree $m-j$  
with respect to $\xi$. 
\end{defi}

Given a symbol $a$, we define
the paradifferential operator $T_a$ by
\begin{equation}\label{eq.para}
\widehat{T_a u}(\xi)=(2\pi)^{-d}\int \chi(\xi-\eta,\eta)\widehat{a}(\xi-\eta,\eta)\psi(\eta)\widehat{u}(\eta)
\, d\eta,
\end{equation}
where
$\widehat{a}(\theta,\xi)=\int e^{-ix\cdot\theta}a(x,\xi)\, dx$
is the Fourier transform of $a$ with respect to the first variable,  
$\chi$, $\psi$ are two fixed $C^\infty$ functions such that
$$
\psi(\eta)=0\quad \text{for } \la\eta\ra\le 1,\qquad
\psi(\eta)=1\quad \text{for }\la\eta\ra\geq 2,
$$
 $\chi(\theta,\eta)$ is 
homogeneous of degree $0$ and satisfies, for $0<\eps_1<\eps_2$ small enough,
$$
\chi(\theta,\eta)=1 \quad \text{if}\quad \la\theta\ra\le \eps_1\la \eta\ra,\qquad
\chi(\theta,\eta)=0 \quad \text{if}\quad \la\theta\ra\geq \eps_2\la\eta\ra.
$$ 

\smallbreak

We shall use quantitative results from M\'etivier \cite{MePise} about operator norms estimates in symbolic calculus. 
To do so we introduce the following semi-norms.
\begin{defi}
For $m\in\xR$, $\rho\ge 0$ and $a\in \Gamma^m_{\rho}(\xR)$, we set
\begin{equation}\label{defi:norms}
M_{\rho}^{m}(a)= 
\sup_{\la\alpha\ra\le \frac{d}{2}+1+\rho ~}\sup_{\la\xi\ra \ge 1/2~}
\lA (1+\la\xi\ra)^{\la\alpha\ra-m}\partial_\xi^\alpha a(\cdot,\xi)\rA_{W^{\rho,\infty}(\xR)}.
\end{equation}
\end{defi}

The main features of symbolic calculus for paradifferential operators are given by the following theorems.
\begin{defi}\label{defi:order}
Let $m\in\xR$.
An operator $T$ is said of order $\le m$ if, for all $\mu\in\xR$,
it is bounded from $H^{\mu}(\xR)$ to $H^{\mu-m}(\xR)$.
\end{defi}
\begin{theo}\label{theo:sc0}
Let $m\in\xR$. If $a \in \Gamma^m_0(\xR)$, then $T_a$ is of order $\le m$. 
Moreover, for all $\mu\in\xR$ there exists a constant $K$ such that
\begin{equation}\label{esti:quant1}
\lA T_a \rA_{H^{\mu}\rightarrow H^{\mu-m}}\le K M_{0}^{m}(a).
\end{equation}
\end{theo}

\begin{theo}[Composition]\label{theo:sc}
Let $m\in\xR$ and $\rho>0$. 
If $a\in \Gamma^{m}_{\rho}(\xR)$ and $ b\in \Gamma^{m'}_{\rho}(\xR)$ then 
$T_a T_b -T_{a\# b}$ is of order $\le m+m'-\rho$, where
$$
a\# b=
\sum_{\la \alpha\ra < \rho} \frac{1}{i^{\la\alpha\ra} \alpha !} \partial_\xi^{\alpha} a \partial_{x}^\alpha b.
$$
Moreover, for all $\mu\in\xR$ there exists a constant $K$ such that
\begin{equation}\label{esti:quant2}
\lA T_a T_b  - T_{a\# b}   \rA_{H^{\mu}\rightarrow H^{\mu-m-m'+\rho}}\le 
K M_{\rho}^{m}(a)M_{\rho}^{m'}(b).
\end{equation}
\end{theo}

If $a=a(x)$ is a function of $x$ only, the paradifferential operator $T_a$ is a called a paraproduct. 
Paraproducts can also be defined using the Littlewood-Paley decomposition of the frequency space. 
Indeed, let $\phi\colon \xR\rightarrow \xR$ be a smooth even function with
$\phi(t)=1$ for $\la t\ra\le 1$ and $\phi(t)=0$ for $\la t\ra\geq 2$.
For $k\in \xN$, we introduce the symbol
$$
\phi_k(\xi)=\phi \Big(\frac { \xi}{2^{k}}\Big),
$$
and then the operators $S_k$ and $\Delta_k$ defined by
$$
\widehat{S_{k} f}(\xi)\defn \phi_{k}(\xi)\widehat{f}(\xi), \quad
\widehat{\Delta_{k} f}(\xi)\defn \left(\phi_{k}(\xi) -\phi_{k-1}(\xi)\right)\widehat{f}(\xi)
$$
For all   $f\in\mathcal{S}'(\xR)$, the spectrum of $\Delta_k f$ satisfies
${\rm spec}\, \Delta_k f \subset \{ \xi\,:\, 2^{k-1}  \leq \la\xi\ra \leq 2^{k+1}\}$.
Hence $\Delta_j \Delta_k = 0$ if $\vert j-k \vert \geq 2$. Moreover we have
the Littlewood--Paley decomposition:
$$
f=S_0 f +\sum_{k\in\xN^*}\Delta_k f.
$$
With this decompositon, paraproducts can be defined by 
$$
T_a f =\sum_{k\geq 4}{S}_{k-3}(a)\Delta_k f.
$$
  \\
Notice that the difference between paraproducts defined in these two ways is a smoothing operator. 
Namely, if $a\in W^{\rho,\infty}(\xR)$ for some $\rho >0$ then the difference is of order $-\rho$.  
\begin{theo}\label{ab-}
Let $\alpha, \beta \in \xR$ be such that $ \alpha +  \beta > 0.$ If $a \in H^\alpha(\xR)$ and $ b \in H^\beta(\xR)$ then $ab-T_{a}b - T_{b}a \in H^{\alpha + \beta - \frac{1}{2}}(\xR) $ and
$$\Vert ab - T_ab - T_ba \Vert_{H^{\alpha + \beta - \frac{1}{2}}(\xR)}  \leq K \Vert a \Vert_{H^\alpha(\xR)} \Vert b \Vert_{H^\beta(\xR)}$$
for some positive constant $K$ independent of $a,b.$
\end{theo}
We use the following result which is a consequence of \eqref{esti:quant2} with $m=m'=0,\rho=1.$
\begin{lemm}\label{Tadelta}
Let $s > 2+\frac{1}{2} $ and $a \in  W^{1,\infty}(\xR)$. Then for all $\s \in \xR$ there exists a constant $C>0$ such that for all $j\in \xN$,
$$ \Vert [\Delta_j,T_a]u \Vert_{H^{\s + 1} (\xR)} \leq C \Vert  a \Vert_ {W^{1,\infty}(\xR)} \Vert u \Vert_{H^\s(\xR)}.$$
\end{lemm}


\subsection{The Dirichlet-Neumann operator}
\begin{lemm} \label{Geta}
Let $s> 2 + \mez$ and $1\leq \s \leq s$.  Then there exists an increasing function $C: \xR ^ + \rightarrow \xR^+$ such that for all $(\eta, \psi) \in H^{s+ \mez} ( \xR) \times H^s (\xR)$
$$ \|G(\eta) \psi\|_{H^{\s-1} ( \xR)} \leq C( \|\eta\|_{H^{s+ \mez} ( \xR)} )\| \psi \|_{H^{\s} ( \xR)}.
$$
Furthermore, if  $(\eta, \psi) \in L^\infty (I; H^{s+ \mez} ( \xR) \times H^s (\xR))$ is a solution of~\eqref{system}. Then
\begin{equation}\label{2.6}
\partial_t ( G(\eta) \psi) = G(\eta) ( \partial_t \psi - \mathfrak{B} \partial_t \eta) - \cnx ( V \partial_t \eta)
\end{equation}
where
\begin{equation}\label{BV}
 \mathfrak{B}(t,x) := \frac{ \partial_x \psi \partial_x \eta + G(\eta) \psi} { 1+ |\partial_x\eta| ^2}, \qquad V(t,x) := \partial_x \psi - \mathfrak{B} \partial_x \eta
\end{equation}
\end{lemm}

\subsection{Symmetrization}

We consider a  solution 
$(\eta,\psi)$ of ~\eqref{system} on the time interval $I = [0,T]$ with $0<T<+\infty$, 
 satisfying the assumption \eqref{Ht}   for all $t\in I$ and such that
$$
(\eta,\psi)\in C^0\big(I,H^{s+\mez}(\xR)\times H^{s}(\xR)\big),
$$
for some $s>\frac{5}{2}$. Then we set
\begin{equation}
U = \psi - T_\mathfrak{B}\eta.
\end{equation}
where $\mathfrak{B}$ has been defined in \eqref{BV}.
It follows from  the analysis in \cite{ABZ} that we have the following symmetrization of 
the equations.

\begin{lemm}[{\cite[Corollary 4.9]{ABZ}}]
Let $c$, $c_1$ be defined by
\begin{equation*}
c=\left(1+(\partial_x \eta )^2\right)^{-\frac{3}{4}},\quad 
c_1=\left(1+(\partial_x \eta )^2\right)^{-\mez}.
\end{equation*}
There exists an elliptic symbol $p\in \Sigma^{1/2}_{s-1}$ such that  the complex-valued unknown 
\begin{equation}\label{d.Phi}
\Phi=T_{p} \eta +i T_{c_1}U
\end{equation}
satisfies a  scalar equation of the form
\begin{equation}\label{PhiF}
\partial_{t}\Phi   +T_V\partial_x \Phi  +i \la D_x\ra^\tq T_c \la D_x\ra^\tq \Phi =F,
\end{equation}
where $V$ has been defined in \eqref{BV} and  $F\in L^{\infty}(I,H^{s}(\xR))$. 
\end{lemm}



\section{Reductions}

\subsection{Change of variables}
Our aim in this section is to simplify the equation \eqref{PhiF} by a change of variable.
 To compute the effect of a change of variable 
we shall use Alinhac's paracomposition operators and
we refer to~\cite{Alipara} for the general theory  .
 
 Let $\kappa$ be  a $C^1$ diffeomorphism from $\xR$ to $\xR$. 
We  define the operator $\kappa^*$   by,
\begin{equation}\label{defkappa}
\kappa^* u =u\circ \kappa - T_{(\partial_x u) \circ \kappa}\kappa.
\end{equation}
One of the main properties of $\kappa^*$ is that there is a symbolic calculus theorem which
allows to compute the equation satisfied by~$\kappa^* u$ in terms of the equation satisfied by $u$
(in analogy with the paradifferential calculus). 

\begin{theo}\label{theo:paracomp2}
Let $ d \geq 1, m\in\xR$, $r>1$, $\rho>0$ and set $\sigma\defn \inf\{\rho,r-1\}$. 
Consider a $C^{r}(\xR^d)$-diffeomorphism $\chi$ and set $ \kappa = \chi^{-1}.$
 Let a be symbol in $\Sigma^{m}_{\rho}(\xR^d)$. 
Then there exists $a^*\in \Sigma^{m}_{\sigma}(\xR^d)$ such that
$$
\kappa^* T_{a}  - T_{a^*} \kappa^* \quad\text{is order }\le m- \sigma.
$$
Moreover one can give an explicit formula for $a^*$.
If   $a = \sum a_{m-k}$, then
\begin{equation}\label{defi*}
a^*(\chi(x),\eta)=\sum_{\alpha}  
\frac{1}{i^ {\vert \alphaÊ\vert} \alpha !}
\partial_{\xi}^{\alpha} a_{m-k}(x,{}^t\chi'(x) \eta) \partial_{y}^\alpha (e^{i\Psi_{x}(y)\cdot\eta})\arrowvert
_{y=x},
\end{equation}
where the sum is taken over all $\alpha\in \xN^d$ such that the summand is well defined, 
$\chi'(x)$ is the differential of $\chi$, $t$ denotes transpose and
\begin{equation}\label{defiPsi}
\Psi_{x}(y)=\chi(y)-\chi(x)-\chi'(x)(y-x).
\end{equation}
\end{theo}


 We note that it is easy to obtain regularity results on 
$u$ given results on $\kappa^* u$. Namely, we have the following lemma.

\begin{lemm}
Let $\rho\ge 0, \kappa\in W^{\rho,\infty}(\xR)$ and $ u\in W^{1,\infty}(\xR)$. If $\kappa^*u\in W^{\rho,\infty}(\xR)$ then $u\in W^{\rho,\infty}(\xR)$.
\end{lemm}
\begin{proof}
This follows from the fact that 
$u=(\kappa^* u)\circ \chi +(T_{(\partial_x u)\circ\kappa)}\kappa)\circ \chi$ 
where $\chi=\kappa^{-1}\in W^{\rho,\infty}(\xR)$.
\end{proof}


We are now ready to simplify \eqref{PhiF}. Define $\chi$ by 
 \begin{equation}\label{CHI}
\chi(t,x)=\int_0^x c(t,y)^{-\frac{2}{3}}\, dy = \int_0^x\sqrt{1+(\partial_y \eta (t,y))^2}\, dy,
\end{equation}
so that
$$
\partial_x\chi(t,x)= \sqrt{1+(\partial_x \eta (t,x))^2} = c(t,x)^{-\frac{2}{3}}.
$$
Then for each $t\in [0,T]$, $x\mapsto \chi(t,x)$ is a diffeomorphism from $\xR$ to $\xR$. 
Introduce its inverse 
\begin{equation}
\kappa=\chi^{-1}.
\end{equation}

 \subsubsection{Notations:} 
  We shall set $I = [0,T]$ and we shall denote
\begin{equation}
A = C\big( \Vert (\eta, \psi) \Vert_{L^\infty (I, H^{s+ \mez}(\xR) \times H^{s}(\xR)} \big)
\end{equation}
where $C: \xR^+ \to \xR^+$  is an increasing function which may change from line to line. Moreover we shall denote by $f \circ \kappa$ the function
\begin{equation}\label{frondkappa}
(f \circ \kappa) (t,x) = f(t,\kappa(t,x)).
\end{equation}
\subsubsection{Estimates of $\chi$ and $\kappa$.}
From \eqref{CHI}, the equation $\partial_t \eta = G(\eta) \psi$, the Lemma \ref{Geta}, the H\"older inequality and the fact that $s> 2+ \mez$ we deduce,

\begin{equation}\label{chi_t}
\Vert \partial_t \chi \Vert_{L^\infty (I \times \xR)} \leq A.
\end{equation}
Now since 
$$ \partial_x \chi (t,x) = 1 + f(\partial_x \eta), \quad f\in C^\infty(\xR), f(0) = 0,$$
we deduce from the assumption $s>2+ \mez$ and the Sobolev embedding that,
\begin{equation}\label{chi_x}
 \Vert \partial_x \chi (t,x) - 1\Vert_{L^\infty(I,H^{s- \mez}(\xR))} +
\Vert \partial_x \chi \Vert_{L^\infty (I \times \xR)} \leq A.
\end{equation}
Let us consider the function $\kappa$.

Since $\partial_t \kappa = - \frac{\partial_t \chi}{\partial_x \chi} \circ \kappa$ we have, using \eqref{chi_t},
\begin{equation}\label{kappa_t}
 \Vert \partial_t \kappa \Vert_{L^\infty (I \times \xR)} \leq A.
\end{equation}

On the other hand we have $\partial_x \kappa =  1 + f(\partial_x \eta)$ where $ f\in C^\infty(\xR), f(0) = 0.$ It follows that,
\begin{equation} \label{kappaHs}
\Vert \partial_x \kappa - 1 \Vert_{L^\infty(I, H^{s-\mez}(\xR))} \leq A.
\end{equation}

It is clear from the definition that we have,
\begin{equation}\label{kappa_x}
\vert \partial_x \kappa (t,x) \vert \leq 1, \quad \forall (t,x) \in I\times \xR.
\end{equation}
It follows then by induction that for every $p \in \xN$ we have,
\begin{equation}\label{kappaWp}
\Vert \kappa \Vert_{L^\infty(I, W^{p, \infty}(\xR))} \leq C \big( \Vert \eta \Vert_{L^\infty(I, W^{p, \infty}(\xR))}\big).
\end{equation}
To go further we shall need the following elementary lemma.
\begin{lemm} \label{urondkappa}
Let $p \in \xN^*$ and $ \kappa : \xR \to \xR $ be  a diffeomorphism such that $\partial_x\kappa \in W^{p-1,\infty} (\xR)$. Set  $\chi = \kappa ^{-1}.$ Then for all  $F \in H^{\mu} (\xR) $ with $  0 \leq \mu \leq p$ we have $F \circ \kappa \in H^{\mu} (\xR)$ and
\begin{equation}
\Vert F \circ \kappa \Vert_{H^{\mu}(\xR)} \leq \Vert \chi' \Vert_{L^{\infty} (\xR)} C\big( \Vert \partial_x \kappa \Vert_{W^{p-1, \infty}(\xR)} \big ) \Vert F \Vert_{H^{\mu}(\xR)}
\end{equation}
where $C $ is an increasing function from $\xR^+$ to $\xR^+$.
\end{lemm}
Now in our case for (almost all) fixed $t$ and  all $\eps >0$ we have,
\begin{equation}
\Vert\partial_x \kappa(t,\cdot) \Vert_{W^{p-1, \infty}(\xR)} \leq C(\Vert \eta(t,\cdot) \Vert_{W^{p, \infty}(\xR)}) \leq C(\Vert \eta(t,\cdot) \Vert_{H^{p+\frac{1}{ 2} + \eps}(\xR)}). 
\end{equation}
  We deduce then from Lemma \ref{urondkappa} that for $0 \leq \mu \leq s-1$ and $ F \in L^\infty(I, H^\mu (\xR))$ we have,
\begin{equation}\label{FKappa}
\Vert F  \circ \kappa \Vert_{L^\infty(I, H^{\mu}(\xR))} \leq   A \Vert F \Vert_{L^\infty(I, H^{\mu}(\xR))}.
\end{equation}
Coming back to the regularity of $\chi$ we deduce from \eqref{CHI}  that,
 $$ \partial_x^2 \chi = \frac{(\partial_x \eta)(\partial^2_x \eta)}{(1+(\partial_x \eta)^2)^\mez}.$$
 It follows from \eqref{FKappa} that,
 \begin{equation}\label{chi_xx}
 \Vert ( \partial_x^2 \chi) \circ \kappa \Vert_{L^\infty(I, H^{s-\frac 3 2 }(\xR))} \leq A.
 \end{equation}
 On the other hand we have,$$
  \partial_x \partial_t \chi = \frac{(\partial_x \eta)\partial_x (G(\eta)\psi)}{(1+(\partial_x \eta)^2)^\mez}.$$
 So using Lemma \ref{Geta} and \eqref{FKappa} we obtain,
 \begin{equation}\label{chi_tx}
 \Vert (\partial_x\partial_t  \chi) \circ \kappa \Vert_{L^\infty(I, H^{s-2}(\xR))} \leq A.
 \end{equation}
 Now we would like to estimate $\partial^2_t \chi$. Since $\partial_t \eta = G(\eta) \psi$ we have,
\begin{align}\label{K}
\partial^2_t \chi(t, x) &= - \int_0^x \frac{[\partial_x \eta \partial_x(G(\eta)\psi)]^2}{(1+(\partial_x\eta)^2)^{ \frac{3}{2}}} dy+ \int_0^x \frac{[\partial_x(G(\eta)\psi)]^2}{ (1+(\partial_x\eta)^2)^{ \frac{1}{2}}} dy\\
&+  \int_0^x \frac{ \partial_x \eta \partial_x \partial_t(G(\eta)\psi)}{ (1+(\partial_x\eta)^2)^{ \frac{1}{2}}}dy.
\end{align}
 Since $s>2+ \mez$, the H\" older inequality and Lemma \ref{Geta}  show that the first two terms are pointwise bounded by $A.$
 By the Holder inequality the last term can be pointwise bounded by 
 $$ \Vert \partial_x \eta \Vert_{L^\infty(I, L^2(\xR))} \Vert \partial_x \partial_t\big (G(\eta)\psi\big )\Vert_{L^\infty(I, L^2(\xR))}.$$ 
 
 Using \eqref{2.6} and the equation satisfied by $(\eta,\psi)$ we find,  if $s>3+ \mez$,
\begin{equation*}
 \Vert \partial_x \partial_t \big (G(\eta)\psi \big )\Vert_{L^\infty(I, L^2(\xR))} \leq A.
\end{equation*}
Therefore if  $s>3+ \mez$ we obtain,
\begin{equation}\label{chi_tt}
\Vert \partial^2_t \chi \Vert_{L^\infty (I \times \xR)} \leq A. 
\end{equation}
Finally let us estimate the term $\partial_x \partial^2_t \chi$.  Using again \eqref{2.6} and the equation satisfied by $(\eta,\psi)$ we find, if $s>4$,  that,
\begin{equation}\label{chi_xtt}
\Vert \partial_x \partial^2_t \chi \Vert_{L^\infty(I, H^{s-\frac{7}{2}}(\xR))} \leq A.
\end{equation}

\subsubsection{Reduction of the equation}
With  $V$ defined in \eqref{BV} and $\Phi$ defined in \eqref{d.Phi}  let us  set (see \eqref{frondkappa}),
\begin{equation}\label{eq.W}
W = V\circ\kappa (\partial_x\chi \circ\kappa )+\partial_t \chi\circ\kappa,
\end{equation}
\begin{equation}\label{d.Phi*}
\Phi^*=\kappa^* \Phi =\Phi\circ\kappa  - T_{(\partial_x \Phi) \circ\kappa }\kappa.
\end{equation}

 Then we have the following result.
\begin{prop}\label{EQu}
Let $s>2+ \mez$ and $I=[0,T].$ There exists a real valued function $g$ such that $\partial_x g \in \Sigma^0_{s-\frac{3}{2}}$ and the function  $u=T_{e^{ig}}\Phi^*$ satisfies the equation
\begin{equation}\label{Phif}
(\partial_{t}  +T_W\partial_x +i \la D_x\ra^\tdm )u=F,
\end{equation}
with $F\in L^{\infty}(I,H^{s}(\xR^d))$ and $W$ is defined by \eqref{eq.W}. 
\end{prop}

\begin{proof}
We apply the operator $\kappa^*$ to the equation \eqref{PhiF}.
We first show that
\begin{equation}\label{estR}
\begin{aligned}
&\kappa^* (\partial_t +T_V \partial_x)\Phi = (\partial_t +T_{W} \partial_x)\Phi^* + R(\Phi)\\ 
&\Vert R(\Phi) \Vert_{L^\infty (I,H^s(\xR))} \leq  C( \Vert (\eta, \psi) \Vert_{L^\infty (I, H^{s+ \mez}(\xR) \times H^{s}(\xR)})\Vert \Phi \Vert_{L^\infty (I,H^s(\xR))}.
\end{aligned}
\end{equation}
We begin by showing that
\begin{equation}\label{estR1}
\kappa^* (\partial_t \Phi)= (\partial_t - T_{(\partial_t \chi) \circ \kappa}) \Phi^* + R_1(\Phi)
\end{equation}
where $R_1$ satisfies the estimate in \eqref{estR}.

We have
\begin{equation*}
\begin{aligned}
\kappa^*(\partial_t \Phi) &= (\partial_t \Phi) \circ \kappa - T_{(\partial_x \partial_t \Phi ) \circ \kappa} \kappa\\
&= \partial_t(\Phi \circ \kappa) - (\partial_t \kappa)(\partial_x\Phi\circ \kappa) -T_{(\partial_x \partial_t \Phi )\circ \kappa} \kappa,
\end{aligned}
\end{equation*}
therefore,
\begin{equation}\label{kappa*_1}
\begin{aligned}
\kappa^*(\partial_t \Phi)& = \partial_t (\kappa^* \Phi) + B_1 + B_2,\\
B_1& =T_{(\partial_x^2 \Phi \circ \kappa) \partial_t \kappa } \kappa\\
B_2&  =  T_{(\partial_x \Phi ) \circ \kappa} \partial_t\kappa  - (\partial_t \kappa)(\partial_x\Phi\circ \kappa)
\end{aligned}
\end{equation}
Let us consider the term $B_1$ in \eqref{kappa*_1} and let us set $a = \partial_t\kappa(\partial^2_x \Phi \circ \kappa).$ We have,
$$T_a \kappa = \sum_{j \geq 4} S_{j-3}(a) \Delta_j(\kappa) =  \sum_{j \geq 4}2^{-j} S_{j-3}(a) \widetilde{\phi}(2^{-j}D)(\partial_x\kappa) = \sum_{j \geq 4}Êg_j$$
where $\widetilde{\phi} \in C^\infty(\xR), \supp \widetilde{\phi} \subset \{ \mez \leq \vert \xi \vert \leq 2 \}.$ Since $\partial_x \kappa = 1+ f(\partial_x \eta)$ with $f(0) = 0,$ we have $ \widetilde{\Delta}_j(\partial_x\kappa)= \widetilde{\Delta}_j(f(\partial_x \eta))$ so,
$$\Vert g_j \Vert_{L^2(\xR)} \leq 2^{-j} \Vert a \Vert _{L^\infty(\xR)} 2^{-j(s-\mez)} c_j C(\Vert \eta \Vert_{H^{s + \mez}(\xR)}), \quad (c_j) \in l^2.$$
On the other hand  using  \eqref{chi_t}, \eqref{chi_x} we can write,
\begin{equation*}
\begin{aligned}
 \Vert a \Vert _{L^\infty(I \times \xR)} &\leq \Vert \Phi \Vert_{L^\infty(I, H^s(\xR))} \Vert \partial_t\chi (\partial_x \chi)^{-1}\Vert_{L^\infty(I \times \xR)}\\ 
 &\leq \Vert \Phi \Vert_{L^\infty(I,  H^s(\xR))} C(\Vert (\eta, \psi) \Vert_{L^\infty (I, H^{s+ \mez}(\xR) \times H^{s}(\xR)}).
 \end{aligned}
 \end{equation*}
 It follows that, 
 \begin{equation}\label{A1}
B_1 \leq  C(\Vert (\eta, \psi) \Vert_{L^\infty (I, H^{s+ \mez}(\xR) \times H^{s}(\xR)})  \Vert \Phi \Vert_{L^\infty(I,H^s(\xR))}.
\end{equation}
Let us consider the term $B_2.$

We have, $\partial_t \kappa =  ab$ where $ a=\partial_t \chi \in \Gamma ^0_1, b=\partial_x\kappa \in \Gamma ^0_1.$ It follows from Theorem \ref{theo:sc} that $a\#b = ab$ and $T_{ab} -T_aT_b$ is of order $-1$. Let us set
\begin{equation}
B_{21} = \Vert(T_{ab}-T_aT_b)(\partial_x \Phi\circ \kappa)\Vert_{L^\infty(I, H^s(\xR))}.
\end{equation} 
Using \eqref{esti:quant2} we obtain,
\begin{equation*}
B_{21} \leq \Vert \partial_t \chi\Vert _ {L^\infty(I, W^{1,\infty}(\xR))}
 \Vert \partial_x \kappa\Vert _ {L^\infty(I, W^{1,\infty}(\xR))}  \Vert (\partial_x  \Phi\circ \kappa)\Vert_{L^\infty(I, H^{s-1}(\xR))}.
\end{equation*}
Since $s- \frac{3}{2} > 1$, using \eqref{FKappa} with $\mu = s -1$ we obtain, 
\begin{equation}\label{A21}
B_{21} \leq C\big( \Vert \eta \Vert_{L^\infty(I, H^{s+\mez}(\xR))} \big )\Vert \Phi \Vert_{L^\infty(I, H^{s}(\xR))}.
\end{equation}
Therefore using \eqref{kappa*_1}, \eqref{A1}, \eqref{A21} we obtain,
\begin{equation}
\kappa^* (\partial_t \Phi)= \partial_t \kappa^* \Phi + T_{\partial_t \chi}  T_{\partial_x \kappa} \partial_x \Phi \circ \kappa + R_2(\Phi),
\end{equation}
where $R_2$ satisfies \eqref{estR}.

Now let us set 
$$a = \partial_x \kappa \in L^\infty(I, H^{s- \mez}(\xR)), b= \partial_x \Phi \circ \kappa \in  L^\infty(I, H^{s-1}(\xR)).$$
 It follows from \eqref{ab-} that
\begin{equation}
  \Vert ab - T_{a}b - T_{b}a \Vert_{L^\infty(I, H^{2s-2}(\xR))} \leq \Vert a \Vert_{L^\infty (I, H^{s- \mez}(\xR))}\Vert b \Vert_{ L^\infty(I, H^{s-1}(\xR))}.
\end{equation}
Therefore we obtain
\begin{equation}
\kappa^* (\partial_t \Phi)= \partial_t (\kappa^* \Phi) - T_{\partial_t \chi} \partial_x( \Phi \circ \kappa)+ T_{\partial_t \chi}T_{\partial_x\Phi \circ  \kappa}\partial_x \kappa + R_3,
\end{equation}
where $R_3$ satisfies \eqref{estR}.
Using \eqref{defkappa} we obtain
\begin{equation*}
\kappa^* (\partial_t \Phi)= (\partial_t  - T_{\partial_t \chi} \partial_x)((\kappa^* \Phi)- T_{\partial_t \chi} \partial_x (T_{\partial_x\Phi \circ  \kappa})+T_{\partial_t \chi}T_{\partial_x\Phi \circ \kappa} \partial_x \kappa + R_3,
\end{equation*}
where $R_3$ satisfies \eqref{estR}.

It follows that
\begin{equation*}
\kappa^* (\partial_t \Phi)= (\partial_t  - T_{\partial_t \chi} \partial_x)((\kappa^* \Phi)- T_{(\partial_x^2 \Phi \circ \kappa) \partial_x \kappa} \kappa  + R_3.
\end{equation*}
Now the term $ T_{(\partial_x^2 \Phi \circ \kappa) \partial_x \kappa} \kappa $ can be estimated exactly by the same method as the term $B_1$, therefore we obtain
\begin{equation*}
\kappa^* (\partial_t \Phi)= (\partial_t  - T_{\partial_t \chi} \partial_x)(\kappa^* \Phi)   + R_4,
\end{equation*}
where $R_4$ satisfies \eqref{estR}. This is precisely  \eqref{estR1}.

Now we claim that
\begin{equation}\label{TV}
\kappa^*(T_V\partial_x \Phi) = T_{(V \partial_x \chi) \circ \kappa} \partial_x\kappa^*  \Phi + R_5(\Phi),
\end{equation}
where $R_5$ satisfies \eqref{estR}.
But this is precisely a consequence of Theorem \ref{theo:paracomp2}. Indeed we have for (almost all) fixed t, $a(x,\xi) = iV(t,x)\xi \in \Sigma^1_{s-1},$ and the diffeomorphism $\kappa$ is in $W^{s-\frac{3}{2}}(\xR)$, so $\sigma = s-\frac{3}{2}$ and the remainder term is of order less than $1-(s-\frac{3}{2}) = \frac{5}{2} -s <0.$ Then \eqref{estR} follows from \eqref{estR1} and \eqref{TV}.

Let us consider now the principal part. Applying again  Theorem \ref{theo:paracomp2} we find that,
$$ \kappa^*(\vert D_x \vert^{\frac{3}{4}} T_c \vert D_x \vert^{\frac{3}{4}} \Phi) = \vert D_x \vert^{\frac{3}{2}}\kappa^* \Phi +T_a \kappa^* \Phi,$$
where $a$ is of order $\mez.$
  
Finally, it remains to reduce to the case where $a=0$. Indeed, let $g$ be a real-valued  symbol such that $\partial_x g \in \Gamma^{0}_{s-3/2}(\xR)$ and 
$$\{|\xi|^{3/2}, g\}= -a,$$
then 
if we set
\begin{equation}\label{d.u}
u=T_{e^{ig}}\Phi^*,
\end{equation}
we obtain by symbolic calculus that $u$ satisfies 
\begin{equation*}
(\partial_{t}  +T_W\partial_x +i \la D_x\ra^\tdm +iT_a + T_b)u=F,
\end{equation*}
with $F\in L^{\infty}(I,H^{s}(\xR^d))$ and $b = i\{|\xi|^{3/2}, g\}$. This completes the proof of Proposition \ref{EQu}.
\end{proof}
\subsubsection{Regularity of W}
 The following result gives some informations on the function $W$ defined in \eqref{eq.W}.
\begin{lemm}\label{defW}
 Let $I=[0,T]$, $E= L^\infty ( I\times \xR)$, $F= L^\infty ( I, H^{s-2} ( \xR)).$
\begin{enumerate}
\item  If $s>2+ \mez$, we have  $W\in E$, $\partial _x W \in F$, and
$$ \|W\|_E + \| \partial _xW \|_F \leq C (\|(\eta, \psi)\|_{L^\infty (I,  H^{s+ \mez} ( \xR) \times H^s ( \xR))}).
$$
\item If $s>4$, we have  $\partial _t W,  \partial_x^2 W, \partial_t \partial_x W \in E$ and
$$\|\partial _t W\|_E +\Vert \partial_x^2 W \Vert_E +\Vert \partial_t \partial_x W \Vert_E \leq C (\|(\eta, \psi)\|_{L^\infty (I, H^{s+ \mez} ( \xR) \times H^s ( \xR))}).
$$
\end{enumerate}
\end{lemm}
\begin{proof}
 Let us recall that we have set,
\begin{equation}
A = C\big( \Vert (\eta, \psi) \Vert_{L^\infty (I, H^{s+ \mez}(\xR) \times H^{s}(\xR)} \big)
\end{equation}
where $C: \xR^+ \to \xR^+$  is an increasing function which may change from place to place.\\
Since $s > 2+\mez$ using \eqref{BV} we obtain,
\begin{equation}
\Vert V \Vert_E \leq \Vert \partial_x \psi \Vert_{L^\infty (I, H^{s-2}(\xR))} +\Vert \mathfrak{B}\Vert_{L^\infty (I, H^{s-2}(\xR))}\Vert \partial_x \eta \Vert_{L^\infty (I, H^{s-2}(\xR))} \leq A.
\end{equation}
Then the estimate $\Vert W\Vert_E \leq A$ follows fom \eqref{chi_t} and \eqref{chi_x} .

Now we have
\begin{equation}
\partial_{x}W = \partial_{x}V \circ \kappa+V \circ \kappa(\partial^{2}_{x} \chi \circ \kappa)\partial_{x} \kappa + (\partial_{t} \partial_{x} \chi \circ \kappa) \partial_{x} \kappa.
\end{equation}
Using \eqref{FKappa}  we see that,
\begin{equation}\label{D}
\Vert V \circ \kappa \Vert_F + \Vert \partial_xV \circ \kappa\Vert_F \leq A\Vert V\Vert_{L^\infty(I,H^{s-1}(\xR))} \leq A.
\end{equation}
Now using \eqref{kappaHs}, \eqref{chi_xx} and the fact that $H^{s-2}(\xR)$ is an algebra we deduce,
\begin{equation}
\Vert V \circ \kappa(\partial_x^2 \chi \circ \kappa)\partial_x \kappa \Vert_F \leq A.
\end{equation}
Then the estimate $\Vert \partial_x W \Vert_F \leq A$ follows from \eqref{chi_tx} and \eqref{kappaHs}.

Let us now prove $2.$ We have
\begin{equation}\label{G}
\begin{aligned}
\partial_t W &= \partial_t V \circ \kappa( \partial_x \chi \circ \kappa) +\partial_x V \circ \kappa(\partial_x \chi \circ \kappa) \partial_t \kappa + V \circ \kappa(\partial_t \partial_x \chi \circ \kappa) \\
&+V \circ \kappa( \partial^2_x \chi \circ \kappa) \partial_t \kappa -  \partial^2_t \chi  \circ \kappa - \partial_t \partial_x \chi  \circ \kappa(\partial_x \kappa) =: \sum_{i=1}^ 6B_i.
\end{aligned}
 \end{equation}

It follows from \eqref{FKappa},\eqref{chi_x},\eqref{kappa_t},\eqref{chi_xx},\eqref{chi_tx}, and the Sobolev embedding that
\begin{equation}\label{H}
\vert B_2 \vert + \vert B_3 \vert + \vert B_4 \vert + \vert B_6 \vert \leq A.
\end{equation}
Now we have $ \partial_t V = \partial_x \partial_t \psi - (\partial_t \mathfrak{B}) \partial_x \eta - \mathfrak {B} \partial_x \partial_t \eta.$ So using the equations satisfied by $(\eta, \psi)$, the Sobolev embedding and Lemma \ref{Geta} we obtain
\begin{equation}\label{I}
  \Vert \partial_t V \Vert_{L^\infty (I, H^{s- \frac{5}{2}}(\xR))} \leq A.
\end{equation}
It follows that
\begin{equation}\label{J}
   \vert B_1 \vert  \leq A.
\end{equation}
The term $B_5$ is estimated by $A$ using \eqref{chi_tt}. Therefore using  \eqref{H} and \eqref{J} we deduce that $\Vert \partial_tW \Vert_E \leq A.$

The claim on $\partial_x^2W$ follows from the first part of the Lemma  and the Sobolev embedding since $s > 3+ \mez.$ It remains to consider the quantity $\partial_t\partial_xW.$  We go back to \eqref{G}. The term $\partial_t V \circ \kappa(\partial_x \chi \circ \kappa)$ is bounded by $A$ in $L^\infty(I, H^{s-\frac{5}{2}}(\xR)).$ The third term $V \circ \kappa(\partial_t \partial_x \chi \circ \kappa)$ is bounded by $A$ in $L^\infty(I, H^{s-2}(\xR)).$ The term $\partial_t \partial_x \chi \circ \kappa( \partial_x \kappa) $ is bounded by $A$ in  $L^\infty(I, H^{s-2}(\xR)).$  Therefore  the $\partial_x$ derivative of these three terms are bounded by $A$ in $L^\infty(I, H^{s-\frac{7}{2}}(\xR)).$ By \eqref{chi_t} we have,
\begin{align*}
\Vert \partial_x V \circ \kappa(\partial_x \chi \circ \kappa)\partial_t \kappa \Vert_{L^\infty(I\times \xR)} &\leq A \Vert \partial_x V \circ \kappa(\partial_x \chi \circ \kappa)\Vert_{L^\infty(I\times \xR)}\\
&\leq A\Vert \partial_x V \circ \kappa(\partial_x \chi \circ \kappa) \Vert_{L^\infty(I, H^{s-2})} \leq A.
\end{align*}
 We can apply the same argument for the term $V \circ \kappa( \partial^2_x \chi \circ \kappa) \partial_t \kappa.$ Finally we bound the term $V \circ \kappa(\partial_x\partial^2_{tt}\chi \circ \kappa)$ in the space $L^\infty(I, H^{s-\frac{7}{2}}(\xR))$ by using \eqref{FKappa} and \eqref{chi_xtt}.This completes the proof of our Lemma.

\end{proof}
\subsection{Symbol Smoothing}
In this section we follow an idea of Smith~\cite{Sm} (see also Bahouri-Chemin~\cite{BaCh}), and we are going to smooth out the coefficients of the function $W$ with respect to $x$. As already mentioned, here is the main place where the idea of allowing loss in remainder terms enters.  
We define for $0<\delta \leq 1$,
$$
T_W^\delta = \sum_{j \geq 4} S_{[\delta (j-3)] }(W)\Delta_j,
$$

The key difference between $T_W$ and $T_W^\delta$ is made clear  below.

\begin{lemm} \label{lem.diff}
The operator $T_W-T_W^\delta$ is of order $-\delta (s- \frac 3 2 )$.
\end{lemm}
\begin{proof}
%

Since for almost all fixed t we have, $\partial_x W(t, \cdot) \in H^{s-2}(\xR)$ we have,
\begin{align*}
\lA S_j (W)-S_{[\delta j]}(W)\rA_{L^\infty(\xR)} 
&\le  \sum_{n=[\delta j]}^j \lA \Delta_n (W)\rA_{L^\infty(\xR)}\\
&\le 
K \sum_{n=[\delta j]}^j   2^{-n(s-\tdm)}\le K 2^{-\delta j(s-\tdm)}.
\end{align*}
\end{proof}
In the sequel we shall set
\begin{equation}\label{hWa}
\left\{
\begin{aligned}
h &= 2^{-j}, \, j \in \xN, \\
W_{h}^ \delta &= S_{[\delta(j-3)]}(W),\\
a(\xi) &= \chi_0(\xi) \vert \xi \vert^{\frac{3}{2}},
\end{aligned}
\right.
\end{equation}
where $\chi_0 \in C^\infty_0(\xR), \supp \chi_0 \subset \{ \frac{1}{ 4} \leq \vert \xi \vert \leq 4\},  \chi_0 =1$ in $ \{\frac{1}{ 2} \leq \vert \xi \vert \leq 2\}.$
\begin{lemm}\label{lem.3.5}
 Let $s>2 + \frac 1 2$.   
There exist $\delta<\frac 1 2$, $\epsilon >0$ and $f_h\in L^{\infty}(I,H^{s+\epsilon- \mez }(\xR))$ 
such that 
\begin{equation}\label{eq.estj1}
\begin{gathered}
\|f_h\|_{ L^{\infty}(I,H^{s+\epsilon- \mez}(\xR))} 
\leq C \left(\|(\eta, \psi)\|_{L^\infty (I,H^{s+ \frac 1 2}(\xR)\times H^{s}(\xR))}\right),\\
 \supp (\widehat{f_h})\subset \{\frac1 2 h^{-1}\leq \vert \xi\vert \leq 2 h^{-1}\}
 \end{gathered}
 \end{equation}
and the functions  $u_{h}= \Delta_j u$ satisfy 
\begin{equation}\label{Phifbis}
(\partial_{t}  +\frac1 2 (W^\delta_{h} \partial_x + \partial_x W^\delta_{h}) 
+i a(D_x) )u_{h} =f_h
\end{equation}
\end{lemm}
\begin{proof}
First of all we remark that we have $\la D_x\ra^\tdm u_{h} = a(D_x)u_h.$ Now, applying the operator $\Delta_j$ to~\eqref{Phif}, we obtain
\begin{equation}\label{Phifter}
(\partial_{t}  +T_W \partial_x  +i \la D_x\ra^\tdm) u_{h} =\Delta_j f - [\Delta_j, T_W] \partial_xu := g_h^1
\end{equation}
 Since by Lemma  \ref{defW} we have $\partial_x W \in L^\infty(I, H^{s-2}(\xR))$, it follows from Lemma \ref{Tadelta} that $g_h^1$  satisfies~\eqref{eq.estj1} (for any $0<\epsilon \leq \mez$).
Then we can replace $T_W$ by $T^\delta_W$, which gives
\begin{equation}\label{Phifquar}
(\partial_{t}  +\sum_{|k-j|\leq 1}S_{[\delta(k-3)]}(W)\Delta _k \partial_x  +i \la D_x\ra^\tdm) u_{h} 
=g_h^1+ (T_W^\delta - T_W)\partial _x u_h := g_h^1 + g_h ^2\end{equation}
where, according to Lemma~\ref{lem.diff},  $g_h^2$ satisfies~\eqref{eq.estj1} 
with $\epsilon = \delta (s-\frac 3 2)- \mez>0$ if $\delta<\mez$ is chosen 
close enough to $\mez$.  Now, we have
$$
S_{[\delta(j-3)]}(W) \partial_x u_{h} = \sum_{|k-j| \leq 1} 
S_{[\delta(j-3)]}(W)\Delta _k  \partial_x u_{h}.
$$ 
Consequently, we obtain
\begin{multline}
(\partial_{t}  +S_{[\delta(j-3)]}(W) \partial_x   +i \la D_x\ra^\tdm) u_{h} 
\\=g_h^1+g_h^2+ \sum_{|k-j| \leq 1}
(S_{[\delta(j-3)]}(W)-S_{[\delta(k-3)]}(W)) \Delta _k \partial_x u_{h}
=g_h^1+g_h^2+g_h^3,
\end{multline}
and  using that for $|k-j|\leq 1$,
$$
\lA S_{[\delta(k-3)]}(W)-S_{[\delta(j-3)]}(W)\rA_{L^\infty} \leq C 2^{-j\delta(s- \frac 3 2)},
$$
we obtain that $g_h^3$ satisfies~\eqref{eq.estj1}.  Finally, we obtain
$$ (\partial_{t}  +\mez( W^\delta_{h} \partial_x + \partial _x W^\delta_{h} )   +i \la D_x\ra^\tdm) u_{h} 
\\=g_h^1+g_h^2+ g_h^3+ g_h^4,
$$
where $g_h^4 = \mez S_{[\delta (j-3)]} ( \partial_x W) u_{h}$ satisfies~\eqref{eq.estj1} 
(for any  $0<\epsilon \leq \mez$).
\end{proof}
\begin{lemm}\label{lem.3.6}
Let $s>\frac{11}{2}$ and set    
$$\delta = \frac{1}{s - \frac 3 2} \in ]0, \frac 1 4 [.$$ Then there exists $f_h\in L^{\infty}(I,H^{s}(\xR))$ 
such that 
\begin{equation}\label{eq.estj2}
\begin{gathered}
\|f_h\|_{ L^{\infty}(I,H^{s}(\xR))} 
\leq C \left(\|(\eta, \psi)\|_{L^\infty (I,H^{s+ \frac 1 2}(\xR)\times H^{s}(\xR))}\right),\\
 \supp (\widehat{f_h})\subset \{\frac1 2 h^{-1}\leq \vert \xi\vert \leq 2h^{-1} \}
 \end{gathered}
 \end{equation}
and the functions  $u_{h}= \Delta_j u$ satisfy 
\begin{equation}
(\partial_{t}  +\frac1 2 (W^\delta_{h} \partial_x + \partial_x W^\delta_{h}) 
+i a(D_x)u_{h} =f_h
\end{equation}
\end{lemm}

 \begin{proof} 
The proof is identical to that of Lemma~\ref{lem.3.5}, the only difference beeing that now we  take $\delta $ such that $\delta ( s- \frac 3 2 ) = 1$.
\end{proof}

\section{Semi-classical parametrix}\label{sec.semiclass}
Following \cite{BGT1} we shall reduce the analysis to establishing semi-classical estimates. 

Recall that  $2^{-j} =h$ and $W_h^\delta=S_{[\delta(j-3)]}(W) = \gamma(h^\delta D_x)W$, $\gamma \in C_0^\infty(\xR).$ 
\begin{theo}\label{DISP}
 Let $\chi \in C_0^\infty (\xR ^)$ with $\supp \chi \subset \{ \xi : \frac{1}{2} \leq \vert \xi \vert \leq 2 \}$ and $t_0 \in \xR.$  For any initial data $u_{0,h}= \chi(hD_x) u_0$,  where $u_0 \in L^1(\xR),$ let  
 $U_h := S(t,t_0,h) u_{0,h}  $ be  the solution of
\begin{equation}\label{Phifquint}
\partial_{t}U_h   +\frac1 2 (W^\delta_h \partial_x + \partial_x W^\delta_h) U_h  
+i a(D_x) U_h =0, \qquad U_h \mid_{t=t_0} = u_{0,h}.
\end{equation}
 Then there exists $\varepsilon>0$ such that for any $0<h \leq1$ and any $|t-t_0|\leq  h^{\mez- \varepsilon}$,
\begin{equation}
\label{eq.dispersion1}
\|S(t,t_0,h) u_{0,h}\|_{L^\infty(\xR)} \leq \frac C {h^{1/4} |t-t_0|^{1/2}} \|u_{0,h}\|_{L^1(\xR)}.
\end{equation} 
\end{theo}
To prove this result, we shall follow a very classical trend and construct a parametrix. Notice that our assumptions being time-translation invariant  we can assume $t_0=0$. The parametrix will take the following form,
\begin{equation}\label{parametrix}
\widetilde{U}_h(t,x)= \frac 1 {2\pi h} \iint e^{\frac i h (\Phi(t,x,\xi,h)- z\xi)} \widetilde {B}(t,x,z,\xi,h) u_{0,h}(z) dz d\xi,
\end{equation}
where $\Phi$  will satisfy the eikonal equation and
\begin{equation}\label{amplitude}
 \widetilde {B}(t,x,z,\xi,h) = B(t,x,\xi,h) \zeta(x-z-th^{- \mez}a'(\xi)),
 \end{equation}
where $B$ will satisfy the transport equations and $ \zeta \in C_0^\infty (\xR)$, $\zeta(s) = 1$ if $\vert s \vert \leq 1$, $\zeta (s) = 0$ if $\vert s \vert \geq 2.$

 In addition to $\chi$ we shall use two more cut-off functions $\chi_j \in C_0^\infty (\xR), j = 1,2$,  such that
 \begin{equation}\label{chi_0chi_1}
 \left\{
 \begin{aligned}
  &\supp \chi_1 \subset \{ \xi : \frac{1}{3} \leq \vert \xi \vert \leq 3 \}, \quad \chi_1 = 1\,  \text{on}\,   \supp  \chi, \\
  &\supp \chi_0 \subset \{ \xi : \frac{1}{4} \leq \vert \xi \vert \leq 4 \}, \quad\chi_0 = 1 \,\text{on} \, \supp \chi_1.
  \end{aligned}
  \right.
\end{equation}

\subsection{The eikonal and transport equations}
 We introduce some space of symbols in which we shall solve our equations.
\begin{defi}\label{def.symb}
For  small $\eps,h_0$  to be fixed, we introduce the sets
 \begin{equation*}
 \begin{aligned}
 \Omega_ \eps&=\left\{(t,x,\xi,h)\in \xR^4\, :\, h\in (0, h_0), |t| <  h^{\mez -\eps}, 1<\la \xi\ra < 3\right\},\\
 \mathcal{O}_ \eps&=\left\{(\s,x,\xi,h)\in \xR^4\, :\, h\in (0, h_0), |\s| <  h^{ -\eps}, 1<\la \xi\ra < 3\right\}.
\end{aligned}
\end{equation*}
If
$m\in \xR $ and $\varrho\in \mathbb{R}^+$,  we denote by $S^m_{\varrho, \eps}(\Omega_\eps)$ (resp.$ S^m_{\varrho, \eps}(  \mathcal{O}_ \eps)$) the set of all 
functions $f$ on $  \Omega_\eps$ which are $C^\infty$ with respect to 
$(t,x,\xi)$ (resp.$(\s,x,\xi)$) and satisfy the estimate
\begin{equation}\label{C0}
\la \partial_x^\alpha f(t,x,\xi,h)\ra (resp.
\la \partial_x^\alpha f(\s,x,\xi,h)\ra )
\le C_{\alpha}  h^{m-\varrho\alpha}, 
\end{equation}
for all $(t,x,\xi,h)\in \Omega_\eps$ (resp.$(\s,x,\xi,h)\in \mathcal{O}_\eps)$.
\end{defi}
\begin{rema}
$(i)$ If $f\in S^m_{{\varrho,\eps}}$, $g\in S^{m'}_{\varrho,\eps}$ then $fg\in S^{m+m'}_{\varrho,\eps}$; if $f\in S^m_{\varrho,\eps}, (m \geq 0)$ and $F\in C^\infty(\xC)$ then 
 $F(f)\in S^{m}_{{\varrho},\eps}$;
if $f\in S^m_{\varrho,\eps},(m \leq 0)$ and $F\in C^\infty_b(\xC)$ then 
 $F(f)\in S^{0}_{{\varrho}-m,\eps}$  . Let $f\in S^\mu_{\varrho,\eps}$, then 
$\partial_x f\in S^{\mu -{\varrho}}_{{\varrho,\eps}}$.  Moreover $S^m_{\varrho,\eps} \subset S^m_{{\varrho',\eps}}$ if ${\varrho} \geq {\varrho} '$

$(ii)$ Let $W$ be such that $\partial_x W\in H^{s-2}(\xR)$ with $s>2+\mez$ and set $W^\delt_h=\gamma (h^{\delt} D_x )W$ where 
$\gamma \in  {\mathcal S}(\xR)$. Then $\partial_x W^\delt_h\in S^0_{\delt, \eps}$.
\end{rema}

Let $\delta\in (0, \frac 1 2 )$. We fix
 \begin{equation}\label{mu0}
\mu_0 = \frac 1 2 \left( \frac 1 2 - \delta\right), \quad \eps \in (0, \frac {\mu_0} 5).
 \end{equation}

Until the end of this section, for the simplicity of notations we shall drop the index $\eps$ and denote by $S^m_\varrho(\Omega) $ the space $S^m_{\varrho, \eps}(\Omega_\eps)$. 
Finally we set,  
\begin{equation}\label{L}
\begin{aligned}
\mathcal{L}_0&= \partial_t + \frac 1 2 (W^\delta_h \partial_x +\partial_x W^\delta_h) + i\chi_0 (h D_x) \la D_x\ra^\tdm,\\
a(\xi)&=  \chi_0(\xi)\la \xi\ra^\tdm.
 \end{aligned}
\end{equation}
The main result of this section is the following.
\begin{prop}\label{eikampl}
There exist a phase $\Phi$ of the form 
$$\Phi(t,x,\xi,h)= x \xi - h^{-\mez}ta(\xi)+h^{\mez}\Psi(t,x,\xi,h)$$
 with $\partial_x \Psi \in S^{-\eps}_{\delta}(\Omega)$ and an amplitude $B  \in  S^{0}_{\delta +3\eps}(\Omega)$ such that, with $\widetilde B$ defined in \eqref{amplitude},
 \begin{equation}\label{Leiphi}
\mathcal{L}_0 \left( e^{\frac{i}{h}\Phi}\widetilde{B}  \right)= e^{\frac{i}{h}\Phi}R_h.
\end{equation}
    and for all $N\in \xN$ we have,
\begin{equation}\label{estReste}
\Big\Vert \iint e^{\frac{i}{h}(\Phi(t,x,\xi,h) - z\xi)}R_h(t,x,z,\xi,h)u_{0,h}(z)\,dz\, d\xi \Big\Vert_{ 
H^1(\xR_x)} \leq C_Nh^N \Vert u_{0,h} \Vert_{L^1(\xR)},
\end{equation}
for all t in $[0, h^{\mez-\eps}]$.
 \end{prop}
\begin{proof}
 We set,
\begin{equation}\label{tsigma}
\begin{gathered}
t=h^\mez \s, \qquad \varphi(\s,x,\xi,h)=\Phi(\s h^\mez,x,\xi,h),\\
\widetilde{b}(\s,x,\xi,h)=\widetilde{B}(\s h^\mez,x,\xi,h), \qquad V_h(\s,x,h)=W^\delta_h (\s h^\mez,x,h),\\
\mathcal{ L} =    h \partial_\s +  \frac{1}{2}h^\mez \big( V_h  (h \partial_x) + h  \partial_x V_h\big)+ ia (h D_x).\\
  \end{gathered}
\end{equation}
Multiplying \eqref{Leiphi}   by $h^\tdm$ we see that it is equivalent to,
\begin{equation}\label{C2}
 \mathcal{L} \left( e^{\frac{i}{h}\varphi}\widetilde{b} \right)  = e^{\frac{i}{h}\varphi} r(\s,x,z,\xi,h),
\end{equation}
 and \eqref{estReste} becomes,
\begin{equation}\label{estireste}
\Big\Vert \iint e^{\frac{i}{h}(\varphi(\s,x,\xi,h) - z\xi)} r(\s,x,z,\xi,h)u_{0,h}(z)\,dz\, d\xi \Big\Vert_{H^1(\xR_x)} \leq C_N h^N \Vert u_{0,h} \Vert_{L^1(\xR)}.
\end{equation}
 In the proof of \eqref{C2}, $z, \xi, h$ will be considered as  parameters.

  We shall take $\varphi$ of the form
\begin{equation}\label{C3}
\varphi(\s,x,\xi,h)=x\xi - \s a(\xi)+ h^\mez\psi(\s,x,\xi,h),
\end{equation}
where $\psi$ is the solution of the problem
\begin{equation}\label{C4}
\left\{
\begin{aligned}
&\partial_\s\psi+ a'(\xi)\partial_x\psi=-\xi V_h,\\
&\psi\arrowvert_{\s =0}=0.
\end{aligned}
\right.
\end{equation}
  Differentiating \eqref{C4} with respect to  $x$ and $\xi$, using an induction on $k$  and the fact that $ \partial_x V_h \in S_\delta ^0 (\mathcal {O})$, we see easily that,
 \begin{equation}\label{psiksi}
\vert \partial_\xi^k \partial_x ^\alpha \psi (\s,x,\xi,h) \vert \leq C_{k \alpha} \vert \s \vert  h^{-k \eps}  h^{- \delta(\alpha +k - 1)^+}, 
\end{equation}
for every $(\s,x,\xi,h) \in \mathcal{O}$, where $a^+ = \sup(a, 0)$. It follows in particular that $\partial_x \psi\in S^{-\eps}_\delta (\mathcal {O})$, $\partial_\s \psi\in S^{-\eps}_\delta (\mathcal {O})$ .

Now, since $\widetilde{b} = b\,\zeta$ we have,
\begin{equation}\label{unpeudeL}
\begin{aligned}
 e^{-\frac{i}{h}\varphi}\big( h \partial_\s &+  \frac{ h^{\frac{3}{2}}}{2} ( V_h  \partial_x +\partial_x V_h)\big)( e^{\frac{i}{h}\varphi} \widetilde {b}) = i[ h^\mez \xi V_h - a(\xi) +h^\mez \partial_\s \psi + hV_h \partial_x \psi ] \widetilde{b} \\
 &+h [ \partial_ \s b + h^{\frac{1}{2}} V_h \partial_x b  + \frac{1}{2}h^{\frac{1}{2}}(\partial_x V_h) b ] \zeta+ h[-a'(\xi) + h^{\frac{1}{2}} V_h ] b\,\zeta' .
 \end{aligned}
 \end{equation}
On the other hand recall that we have for all $M\in \xN^*$ (see the appendix),
\begin{equation}\label{C6}
e^{-\frac{i}{h}\varphi}a(h D_x) \left( e^{\frac{i}{h}\varphi} \widetilde{b}\right)
=A+r_1+r_2,
\end{equation}
where
\begin{equation}\label{C7}
A=\sum_{k=0}^{M-1}\frac{h^k}{i^k k!} 
\partial_y^k  \left\{ (\partial_\xi^k a) \left( \rho(x,y)\right) \widetilde{b}(y)\right\} 
\Big\arrowvert_{y=x}.
 \end{equation}
with
\begin{equation}\label{C8}
\rho(x,y)=\int_0^1 \frac{\partial \varphi}{\partial x}(\s,\lambda x+(1-\lambda)y,\xi,h ) \, d\lambda,
\end{equation}
and the remainder $r_1,r_2$ are given by,  
\begin{equation}\label{C9}
r_1=c\, h^{M-1} \iiint_0^1 e^{\frac{i}{h}(x-y)\eta}\kappa_0(\eta) 
  (1-\lambda)^{M-1}\partial_y^M 
[a^{(M)}(\lambda \eta+\rho((x,y))\widetilde{b}(y)] d\lambda 
 dy d\eta
\end{equation}
and
\begin{equation}\label{C10}
r_2=\sum_{k=0}^{M-1} c_{k,M} h^{M+k} 
\iint_0^1 z^M \hat{\kappa}_0(z) 
 (1-\lambda)^{M-1}\partial_y^{M+k} [a^{(k)}(\theta)\widetilde{b}] \arrowvert_{y=x-\lambda hz}
  d\lambda dz,
\end{equation}
where $c_M,c_{k,M}\in \xC$, $\kappa_0\in C^\infty_0(\xR), \kappa_0 = 1$ in a neighborhood of the origin.
Now since 
$$\widetilde{b}(\s,x,z,\xi,h)= b(\s,x,\xi,h) \zeta(x-z-\s a'(\xi)),$$
writing for simplicity $b(y) = b(\s,y,\xi,h)$ and $\zeta = \zeta(x-z-\s a'(\xi))$ we have, 
\begin{equation}\label{Anouveau}
\left\{
\begin{aligned}
A &= (\sum_{k=0}^{M-1} A_k) \zeta  + r_3,\\
A_k&=\frac{h^k}{i^k k!}  \partial_y^k  \left\{ (\partial_\xi^k a) \left( \rho(x,y)\right)b(y)\right\}\arrowvert_{y=x},\\
r_3&= \sum_{k=1}^{M-1} \sum_{j=1}^{k} c_{jk} h^k  \partial_y^{k-j}  \left\{(\partial_\xi^k a) ( \rho(x,y))b(y)\right\}\arrowvert_{y=x}\zeta^{(j)}.
\end{aligned}
\right.
 \end{equation}

The term $A_0$ in \eqref{Anouveau} is equal to 
$a(\xi+h^\mez \partial_x \psi)b$. 
Then
\begin{equation}\label{C11}
A_0=\left[\sum_{j=0} ^2 \frac{1}{j!} a^{(j)}(\xi)(h^\mez \partial_x \psi)^j+
\frac{ (h^\mez \partial_x \psi)^3}{2} \int_0^1 (1-\lambda) ^2 \partial_\xi^3 a(\xi+\lambda h^\mez\partial_x\psi)\,d\lambda 
\right] b.
\end{equation}
The term $A_1$ in \eqref{Anouveau} can be written as 
$$
A_1 = \frac{h}{i}\left[ a'(\xi+h^\mez\partial_x\psi) \partial_x b
+\mez h^\mez (\partial_x ^2\psi) a''(\xi+h^\mez\partial_x\psi) 
 b \right].
$$
Therefore
\begin{equation}\label{C12}
\begin{aligned}
A_1=  \frac{h}{i}&\left[ \left\{a'(\xi)+h^\mez \partial_x \psi  \int_0^1 a''(\xi+\lambda h^\mez\partial_x\psi)\,d\lambda \right\}
\partial_x b \right.\\
&\left.
+\mez h^\mez ( \partial_x ^2\psi) a''(\xi+h^\mez\partial_x\psi) 
b\right].
\end{aligned}
\end{equation}

Since $h^\eps\partial_x\psi\in S^{0}_\delta$,  $h^{\eps+\delta} \partial_x^2\psi \in S^0_\delta$ and $5\eps \leq \mu_0$,  we deduce from \eqref{C11} and \eqref{C12} that 
\begin{multline}\label{C13}
A_0+A_1
=\left[a(\xi)+h^\mez a' (\xi)\partial_x \psi +\frac{h}{2} a'' (\xi)(\partial_x\psi)^2\right] b
+\frac{h}{i} a'(\xi)\partial_x b\\
+h h^{\mu_0}h^\eps (c_1 b+ c_2 h^{3\eps+\delta} \partial_x b)
\end{multline}
for some $c_j\in S^0_\delta$. 

Now, consider the term $A_k$ with $k\ge 2$. We have
$$
A_k = \frac{h^k}{i^k k!} \sum_{k_1=0}^k 
\begin{pmatrix} k \\ k_1\end{pmatrix}
\partial_y^{k_1} \left[ (\partial_\xi^k a)(\rho(x,y))\right]\big\arrowvert_{y=x} \partial_x^{k- k_1}b.
$$
Since $h^\mez\partial_x \psi \in S^0_\delta$, we obtain,

$$
c_{k,k_1} :=h^{k_1 \delta}\partial_y^{k_1} \left[ (\partial_\xi^k a)(\rho(x,y))\right]\big\arrowvert_{y=x} \in S^0_\delta.
$$
It follows that 
the generic term in $A_k$ can be written as
\begin{equation*}
h h^{k-1}  h^{-k_1\delta} c_{k,k_1}
h^{-\eps} h^{-(\delta + 3 \eps)(k- k_1)} h^\eps ( h^{\delta+ 3 \eps}\partial_x)^{k-k_1}b.
\end{equation*}
We have, since $3 \eps <\mu_0$ and $k \geq 2$, 
\begin{multline}
k-1 - k_1 \delta - \eps - ( \delta + 3 \eps) (k-k_1)\geq k ( 1 - \delta - 3\eps) -1 - \eps \geq 2( 1 - \delta - 3 \eps) -1 - \eps \\
\geq 2 ( \mez - \delta - 3 \eps) - \eps \geq \mez - \delta - \eps \geq \mu_0
\end{multline}
so that 
$$
A_k=h h^{\mu_0}h^\eps \sum_{\ell=0}^{k}c_\ell (h^{\delta+ 3 \eps} \partial_x )^{\ell} b,\quad 
c_\ell \in S^0_\delta.
$$
 We deduce  from \eqref{C13} that
\begin{equation}\label{C14}
 \sum_{k=0}^{M-1} A_k
=\left[\sum_{j=0} ^2 \frac{1}{j!} a^{(j)}(\xi)(h^\mez \partial_x \psi)^j\right] b 
+\frac{h}{i} a'(\xi)\partial_x b
+h h^{\mu_0}h^\eps \sum_{\ell=0}^{M-1}d_\ell (h^{\delta+3 \eps} \partial_x )^\ell b 
\end{equation}
with $d_\ell \in S^0_\delta$.

 Then it follows from \eqref{C2},\eqref{unpeudeL},\eqref{C6} and \eqref{C14} that 
\begin{align*}
r&=
i\left( -a(\xi) +h^\mez \xi V_h  +h^\mez \partial_\s \psi +h V_h \partial_x \psi\right) b\,Ê\zeta \\
&\quad +h[ \partial_\s b +h^\mez V_h  \partial_x b + \frac1 2h^\mez (\partial_x V_h) b] \zeta\\
&\quad +i\left[a(\xi)+h^\mez  a' (\xi)\partial_x \psi +\frac{h}{2} a'' (\xi)(\partial_x\psi)^2\right] b\, \zeta
+h a'(\xi)\partial_x b\, \zeta\\
&\quad +h h^{\mu_0} h^\eps\sum_{\ell=0}^{M-1}d_\ell (h^{\delta+3 \eps} \partial_x )^\ell b\,\zeta + \sum_{j=1}^3 r_j.
\end{align*}
Gathering the terms in powers of $h$, noting that the coefficient of $h^0$ vanishes 
and using the eikonal equation to see that the coefficient in $h^\mez$ vanishes, 
we are left with
\begin{equation}\label{C16}
r=
h \left( \partial_\s b +a'(\xi)\partial_x b+i f b + h^{\mu_0} h^\eps\sum_{\ell=0}^{M-1}e_\ell (h^{\delta+3\eps} \partial_x )^\ell b
\right)\, \zeta +i\sum_{j=1}^4 r_j.
\end{equation}
where $f =V_h \partial_x\psi + a''(\xi) (\partial_x \psi)^2$ is real-valued,  $e_\ell \in S^0_\delta$ and
\begin{equation}\label{r4}
r_4 = \frac{1}{i}h[-a'(\xi) + h^{\frac{1}{2}} V_h ] b\,\zeta' .
\end{equation} 
It follows from \eqref{psiksi} that
\begin{equation}\label{derf}
\vert \partial_x ^\alpha \partial_ \xi ^k f (\s,x,\xi,h) \vert \leq C_{k \alpha} \s h^{-\eps} h^{-k \eps} h^{-\delta (\alpha +k)},
\end{equation}
for every $(\s,x,\xi,h) \in \mathcal{O}.$ In particular $f \in S^{-2\eps}(\mathcal{O}).$

Now we shall seek $b$ under the form
 \begin{equation}\label{formeb}
b=\sum_{j=0}^{J-1}h^{j\mu_0} b_j.
 \end{equation}
 where the $b_j's$ are the  solutions of the following problems
\begin{equation}\label{C17}
\left\{
\begin{aligned}
&\frac{\partial b_0}{\partial \s} + a'(\xi)\frac{\partial b_0}{\partial x}+i f b_0=0 ,\\
& b_0 \arrowvert_{\s =0}=\chi_{1}(\xi),
\end{aligned}
\right.
\end{equation}
where $\chi_1 \in C_0 ^\infty (\xR) $ has been introduced in \eqref{chi_0chi_1} and
\begin{equation}\label{C18}
\left\{
\begin{aligned}
&\frac{\partial b_j}{\partial \s} + a'(\xi)\frac{\partial b_j}{\partial x}+i f b_j=
-h^\eps\sum_{\ell=0}^{M-1}e_\ell (h^{\delta+3\eps} \partial_x)^\ell b_{j-1} ,\\
& b_j \arrowvert_{\s =0}=0.
\end{aligned}
\right.
\end{equation}
It is easy to see that for all $j$ we have,
\begin{equation}\label{C19}
b_j(\s,x,\xi,h) = \chi_1 (\xi) c_j(\s,x,\xi,h).
\end{equation}

For the estimates we shall use the following elementary lemma.
\begin{lemm}\label{estimations}
If $u$ is a solution of the problem 
$$
\partial_\s u + a'(\xi)\partial_x u +i f u = g,\qquad 
u\arrowvert_{\s =0}=z \in \xC,
$$
where $f$  be real-valued, then it satisfies the estimate
$$ \vert u(\s,x,\xi,h)\vert \leq \vert z \vert + \int_0^\sigma \vert g(\sigma', x + (\sigma'-\sigma)a'(\xi), \xi,h) \vert  d\sigma'$$  
 \end{lemm}
 for every $(\s,x,\xi,h) \in \mathcal{O}.$
 \begin{proof}
 Indeed, the solution is given by 
\begin{multline*}
u(\s,x,\xi,h)=e^{-i\int_0^\s f(\s',x+(\s'-\s)a'(\xi),\xi,h)\, d\s'}\times\\
\times \left\{ z+h^\eps \int_0^\s e^{i\int_0^{\s'} f(t,x+(t-\s)a'(\xi),\xi,h)\, dt}
g(\s',x+(\s'-\s)a'(\xi),\xi,h)\, d\s'\right\}.
\end{multline*}
\end{proof}
 
 Using this lemma we deduce the following.
\begin{lemm}\label{eqb}
The problems \eqref{C17}, \eqref{C18} have unique solutions $b_j = \chi_1(\xi) c_j$  where the $c_j$ satisfy the estimates
\begin{equation}\label{estsymbols}
\vert \partial_x^ \alpha \partial_ \xi ^k c_j (\s,x,\xi,h)\vert \leq C_{\alpha k j} h^{-k \eps} h^{- (\alpha + k) (\delta + 3 \eps)}
\end{equation}
for all $(\s,x, \xi,h) \in \mathcal{O},$ all $ \alpha, k \in \xN,$ and all $j = 0,...,M.$

In particular $c = \sum_{j = 0}^{M} c_j$ belongs to $S^0_{\delta + 3 \eps} (\mathcal{O}).$
\end{lemm}
\begin{proof}
Let us look to the case $ j =0$. Then $c_0$ satisfies the same equation and $c_0 \arrowvert_{\sigma=0} =1.$  We show first \eqref{estsymbols} for $k=0$ and all $\alpha.$ By Lemma \ref{estimations} we have $ \vert c_0 \vert \leq Ch^{- \eps}.$ So assume that \eqref{estsymbols} is true (for $k=0$) up to the order $\alpha -1$ and let us differentiate the equation \eqref{C17} $\alpha $ time with respect to $x.$ It follows that $ U= \partial_x ^\alpha c_0$ satisfies the equation
\begin{equation}\label{equationU}
\frac{\partial U}{\partial \s} + a'(\xi)\frac{\partial U}{\partial x}+i f U = - i \sum_{l=1}^\alpha C^\alpha_l (\partial_x ^l f) \partial_x^{\alpha- l} c_0.
\end{equation}
Using \eqref{derf}, Lemma \ref{estimations} and the induction we deduce that 
\begin{equation*}
 \vert U \vert \leq C h^{- \eps} \sum_{l=1}^\alpha h^{-2 \eps} h^{- l (\delta + 3 \eps)} h^{-( \alpha - l) ( \delta + 3 \eps)} \leq C h^{- \alpha(\delta + 3 \eps)}.
 \end{equation*}
 This proves \eqref{estsymbols} for $k = 0$ and all $\alpha.$ Then using an induction on $k$ we differentiate the equation \eqref{equationU} $k$ times wit respect to $\xi$ we use again \eqref{derf}, Lemma \ref{estimations} and the induction to prove \eqref{estsymbols} for all $k$ and $\alpha.$ The proof of \eqref{estsymbols} for $j \geq 1$ is similar.
\end{proof}
It follows from \eqref{C16}, \eqref{C17}, \eqref{C18} that 
$$
r= \sum_{j=1}^5 r_j.
$$
where
\begin{equation}\label{C20}
r_5 = h^{J\mu_0} b_{J-1} \zeta.
\end{equation}
\subsubsection{End of the proof of Proposition \ref{eikampl}}
We are left with the proof of \eqref{estireste}. For $r_j, j= 1,2,5$ defined in \eqref{C9}, \eqref{C10} and \eqref{C20}  we have for all $N \in \xN$, 
\begin{equation}\label{C21}
 \langle x-z-\s a'(\xi) \rangle \vert r_j(\s,x,z,\xi,h) \vert \leq C_Nh^N \vert \chi_1( \xi)\vert.
 \end{equation}
To  prove \eqref{C21} we write in the integral giving $r_1$ (resp.$r_2$) , $x-z-\s a'(\xi) = (x-y) + (y-z-\s a'(\xi))$, (resp. $ = (x-z - \lambda hz - \s a'(\xi) ) + \lambda hz$) we integrate by parts using the fact that $\frac{h}{i} \partial_\eta e^{\frac{i}{h} (x-y)\eta} = e^{\frac{i}{h} (x-y)\eta}$ and we use the fact that
 \begin{equation}\label{C22} 
 \widetilde{b}(y) = \chi_1(\xi)c(\s,y,\xi,h)\zeta(y-z-\s a'(\xi))
 \end{equation}
 with $c\in S^0_{\delta +3 \eps}$ and $ \delta + 3 \eps < 1.$ Then \eqref{estireste} follows from \eqref{C21} and Young's inequality.
   
 The terms corresponding to $r_3$ and $r_4$ defined in \eqref{Anouveau} and \eqref{r4} will be treated in the same manner and will use the fact that on the support of a derivative of the function $\zeta$ one has $\vert x-z- \s a'(\xi) \vert \geq 1.$  Since by  \eqref{psiksi} we have $h^\mez \vert \partial_\xi \psi \vert \leq Ch^\mez \vert \s \vert \leq C h^{\mez - \eps}$ we deduce from \eqref {C3}  that $\vert \partial_\xi (\varphi(\s,x,\xi,h) - z \xi) \vert \geq \frac{1}{2}$ if $h$ is small enough. Therefore we can integrate N times by parts  using the vector field 
 $$ L = \frac{h}{ i( \partial_\xi (\varphi(\s,x,\xi,h) - z \xi)) } \partial_ \xi. $$ 
 Finally the estimate for the term $r_5$ follows easily from the fact that the convolution of $L^1(\xR)$ with $L^2(\xR)$ is contained in $L^2(\xR)$.
 
   The proof of Proposition \ref{eikampl} is complete.

\end{proof} 

\subsection{Refined Van der Corput estimate}
 Let us recall that we have  set (see \eqref{parametrix})
\begin{equation}\label{1}
  \widetilde{U}_h(t,x) = \int \widetilde{K}(t,x,z,h)u_{0h}(z)\, dz
\end{equation}
where
\begin{equation}\label{2}
\widetilde{K}(t,x,z,h)=\frac{1}{2\pi h} \int e^{\frac{i}{h}(\Phi (t,x,\xi,h)-z\xi)} 
\widetilde{B}(t,x,z,\xi,h)\, d\xi.
\end{equation}
In the variable $\s=t h^{-\mez}$ we have
\begin{equation*}
\widetilde{K}(t,x,z,h)=K(\s,x,z,h)
\end{equation*}
where
\begin{equation*}
K(\s,x,z,h)=\frac{1}{2\pi h} \int e^{\frac{i}{h}(\varphi (\s,x,\xi,h)-z\xi)}\, 
\widetilde{b}(\s,x,z,\xi,h)\, d\xi,
\end{equation*}
where $\varphi$ and $b$ have been determined in \eqref{C3}, \eqref{C17} and \eqref{C18}.
\begin{prop}\label{p4.5}
There exists $C>0$ such that
\begin{equation}\label{3}
\la K(\s,x,z,h)\ra \le \frac{C}{h}\left(\frac{h}{ \s}\right)^{\mez},
\end{equation}
for all $(\s,x,z,h)$ in $]0,h^{-\eps}[\times\xR\times\xR\times ]0,h_0[$.
\end{prop}
\begin{proof}Since $b\in S^{0}_{\delta}$ is bounded with compact support in $\xi$, the estimate 
\eqref{3} is trivial for $\la \s\ra \le Ch$. Let us assume that $\la \s\ra\ge Ch$. 
We have by \eqref{tsigma},
$$
 \mathcal{ L} = \left(h \partial_\s +h^\mez W_h^{\delta}(h^\mez \s,x) (h \partial_x) \frac{1}{2} h^{\frac{3}{2}} (\partial_x W_h^\delta)+ i \la h D_x\ra^\tdm\right) 
$$

By a scaling argument we can assume without loss of generality that $\sigma\geq 1$. Indeed, otherwise, setting 
$$\tau=\frac \sigma {\sigma_0}, \quad \widetilde{x} = \frac x {\sigma_0}, \quad \widetilde{h}= \frac h {\sigma_0},$$
we see that in the new variables, the operator reads
$$ \widetilde{L} = \widetilde{h} \partial_\tau + \widetilde{h} ^{1/2}\widetilde{W} _h^\delta \widetilde{h} \partial _x +          \widetilde{h} ^{\frac{3}{2}} (\partial_{\widetilde{x}} \widetilde{W} _h^\delta) + i | \widetilde{h} D_{\widetilde{x}}| ^{3/2}
$$ where 
$$ \widetilde{W} ^\delta_h ( \tau, \widetilde{x} ) = \sigma_0^{1/2} W_h^\delta ( \sigma_0 \tau, \sigma_0 \widetilde{x})$$
and consequently we have 
$$ \widetilde{W}^\delta_h \in L^\infty(H^{s-1}), \qquad \partial_{\widetilde{x}} \widetilde{W}^\delta_{h} \in S^0_\delta$$ with bounds uniform with respect to $\sigma_0$. 

Assume now that the dispersion estimate has been proved for the kernel of the operator $\widetilde{L}$ and $\sigma= 1$. Since we have 
$$ S_h(\sigma) u_0 (x) = (\widetilde{S} _{\widetilde{h}} ( \frac{ \sigma} {\sigma_0}) \widetilde{u_0}) ( \frac x {\sigma_0}), \qquad \widetilde{u_0} (\frac x {\sigma_0}) = u_0 (x),
$$ 
we can write 
\begin{multline}
\|S_h(\sigma_0)u_0 \|_{L^\infty ( \xR)}= \| \widetilde{S} _{\widetilde{h}} ( \frac{ \sigma} {\sigma_0}) \widetilde{u_0}\|_{L^\infty( \xR)}\\
\leq \frac{ C} { |\widetilde{h}|^{1/2}}\|\widetilde{u_0}\|_{L^1( \xR)}\leq \frac{ C | \sigma_0|^{1/2} } { |h|^{1/2} |\sigma_0|} \|u_0\|_{L^1( \xR)}\leq \frac{ C} { |h \sigma_0|^{1/2}}
\end{multline}
which is the dispersion estimate for the kernel of the operator ${L}$ and $\sigma= \sigma_0$.

Let us set
\begin{equation}\label{4}
\theta(x,y,\xi,h)=\varphi(\s,x,\xi,h)-z\xi=(x-z)\xi+a(\xi)\s+h^\mez \psi(\s,x,\xi,h).
\end{equation}
Then
$$
\partial_\xi^2\theta (x,y,\xi,h)=\partial_\xi^2 a(\xi)\s +h^\mez\partial_\xi^2 \psi(\s,x,\xi,h).
$$
Now by \eqref{chi_0chi_1} and \eqref{L}, on the support of $\chi_1$ we have $a(\xi) = \vert \xi \vert ^{\frac{3}{2}}.$ Therefore  $\partial_\xi^2 a(\xi) = \frac{3}{4} \vert \xi \vert ^{-\frac{1}{2}} \geq c_0 >0.$
 On the other hand from \eqref{psiksi} we have  $ \vert \partial_\xi ^2\psi \vert \leq C \s h^{- \delta - 2 \eps} $  which implies $  h^\mez \vert \partial_\xi^2\psi  \vert \le C \s h^{\mez-\delta - 2 \eps} \leq C\s  h^{\mu_0}.$
It follows that on the support of $\chi_1$ one can find a constant $c_1>0$ such that 
\begin{equation}\label{5}
0<c_1 \s \le \partial_\xi^2\theta (x,y,\xi,h)\le \frac{1}{c_1}\s,
\end{equation}
if $h_0$ is small enough.

In the sequel we shall omit to note $(x,z,h)$ which are fixed taking care of the fact that 
all the constants are independent of $(x,z,h) \in \xR\times\xR\times ]0,h_0[$. 

Let us denote by $[\alpha, \beta] \subset [\frac{1}{3}, 3] $ the support of $\chi_1$.We deduce from \eqref{5} that 
the function $\xi \to \partial_\xi \theta (\xi)$ is increasing on $[\alpha, \beta].$  Therefore one can find $\rho \in [\alpha, \beta]$ such that 
$$ \partial_\xi \theta (\xi)\leq 0 \quad \text{for }  \xi \in  [\alpha, \rho], \quad \partial_\xi \theta (\xi) \geq 0 \quad \text{for }  \xi \in [\rho, \beta].$$
Noting   $b(\s,x,\xi,h) = b(\xi)$ and assuming that  $ ]\rho, \beta[$ is non empty, 
we shall estimate
$$ K_+(\s,x,\xi,h) = \frac{1}{2 \pi h} \int_ \rho^\beta e^{\frac{i}{h} \theta (\xi)} b(\xi) \zeta(x - z - a'(\xi)) d\xi,$$
the estimate corresponding to the intervall $[\alpha, \rho]$ being similar.
We write for small $h$,
\begin{equation}\label{I1I2}
\left\{
\begin{aligned}
K_+ &=  \frac{1}{2 \pi h} (I_1 + I_2),\\
I_1 &= \int _\rho ^{\rho + (\frac{h}{\sigma})^\mez}e^{\frac{i}{h} \theta (\xi)} b(\xi) \zeta(x - z - a'(\xi)) d\xi\\
I_2 & =  \int _{\rho + (\frac{h}{\sigma})^\mez} ^ \beta e^{\frac{i}{h} \theta (\xi)} b(\xi) \zeta(x - z - a'(\xi)) d\xi
\end{aligned}
\right.
\end{equation}
We have obviously,
\begin{equation}\label{estI1}
\vert I_1 \vert \leq C (\frac{h}{\sigma})^\mez
\end{equation}
In the integral $I_2$ using \eqref{5} and the Taylor formula we see that,
\begin{equation}\label{dxitheta}
 \partial_\xi \theta (\xi) \geq c_1 \s \big(\frac{h}{\sigma} \big)^\mez = C_1 (h \sigma)^\mez, \quad \forall \xi \in [\rho + (\frac{h}{\sigma})^\mez, \beta].
 \end{equation}
Let us estimate the integral $I_2.$ We introduce the following notation. We shall  write $A  \Join B$ if $\vert A - B \vert \leq C \big (\frac{h}{\sigma} \big)^\mez$ where C is a constant depending only $q.$ Then we can state the following lemma which is a refined version of the well known Van der Corput Lemma.

\begin{lemm}\label{estI2}
For all $k\in \xN$ we have
\begin{equation}\label{J_k}
 I_2 \Join (-1)^k \big(\frac{h}{i} \big)^k  \int _{\rho + (\frac{h}{\sigma})^\mez} ^ \beta  e^{\frac{h}{i} \theta (\xi)} \frac{1}{(\partial_\xi \theta(\xi))^k }\partial_\xi^k q(\xi) d\xi
 \end{equation}
where  $q(\xi) = b(\xi)\zeta(x-z-a'(\xi)$.
\end{lemm}
\begin{proof}
Let us denote by $J_k$ the  term in the right hand side of \eqref{J_k}. The Lemma is true for $k=0.$ Assume it is true up to the order $k.$ Using the fact that $\frac{h}{i \partial_\xi \theta(\xi)} \partial_\xi e^{\frac{i}{h} \theta (\xi)} = e^{\frac{i}{h} \theta (\xi)}$ and integrating by parts in $J_k$ we obtain,
\begin{equation*}
\begin{aligned}
 J_k &= (-1)^{k+1} \big(\frac{h}{i} \big)^{k +1} \int _{\rho + (\frac{h}{\sigma})^\mez} ^ \beta  e^{\frac{h}{i} \theta (\xi)} \partial_\xi \Big(\frac{1}{(\partial_\xi \theta(\xi))^{k+1} }\Big)\partial_\xi^k q(\xi) d\xi \\
 &+(-1)^{k+1} \big(\frac{h}{i} \big)^{k+1}  \int _{\rho + (\frac{h}{\sigma})^\mez} ^ \beta  e^{\frac{h}{i} \theta (\xi)} \frac{1}{(\partial_\xi \theta(\xi))^{k+1} }\partial_\xi^{k+1} q(\xi) d\xi\\
 &+(-1)^{k+1} \big(\frac{h}{i} \big)^{k +1} [e^{\frac{h}{i} \theta (\xi)} \Big(\frac{1}{(\partial_\xi \theta(\xi))^{k+1} }\Big)\partial_\xi^k q(\xi) ]^{\beta}_{\rho + (\frac{h}{\sigma})^\mez} = J_k^1 + J_k^2 + J_k^3.
\end{aligned}
\end{equation*}
 First of all we have  $J_k^2 = J_{k+1}.$ Now using \eqref{estsymbols} and \eqref{dxitheta} we can write,
$$\vert J_k^3 \vert \leq C h^{k+1}\frac{h^{-k(\delta +4 \eps)}}{(h \sigma)^{\frac{k+1}{2}}} \leq C\big(\frac{h}{\sigma}\big)^\mez h^{k(\mez - \delta - 4 \eps)} \leq C \big(\frac{h}{\sigma}\big)^\mez,$$
since $\sigma \geq 1.$ Now using again \eqref{estsymbols} we obtain,
$$ \vert J_k^1 \vert \leq Ch^{k +1 - k(\delta + 4\eps)} \int _{\rho + (\frac{h}{\sigma})^\mez} ^ \beta \la  \partial_\xi \Big(\frac{1}{(\partial_\xi \theta(\xi))^{k+1} }\Big)\ra d\xi.$$
Since by \eqref{5} the function $\partial_\xi \theta$ is increasing we have 
$$ \la  \partial_\xi \Big(\frac{1}{(\partial_\xi \theta(\xi))^{k+1} }\Big)\ra =  -  \partial_\xi \Big(\frac{1}{(\partial_\xi \theta(\xi))^{k+1} }\Big).$$
Therefore we can write,
$$ \vert J_k^1 \vert \leq Ch^{k +1 - k(\delta + 4\eps)}  [- \Big(\frac{1}{(\partial_\xi \theta(\xi))^{k+1} }\Big)]^{\beta}_{\rho + (\frac{h}{\sigma})^\mez}.$$
 We deduce exactly as for $J_k^3$ that,
$$ \vert J_k^1\vert \leq  C \big(\frac{h}{\sigma}\big)^\mez.$$
It follows that $J_k \Join J_k^2$ which proves our induction.
\end{proof}
Now using Lemma \ref{estI2}, \eqref{estsymbols} and \eqref{dxitheta} we can write,
$$\vert J_k \vert \leq C h^k \frac{1}{(h\sigma)^{\frac{k}{2}}} h^{-k(\delta + 4 \eps)} \leq C \big( \frac{h}{\sigma}\big)^\mez h^{k(\mez - \delta - 4 \eps) - \mez} \leq C\big( \frac{h}{\sigma}\big)^\mez h^{k \mu_0 - \mez} h^{k(\mu_0 - 4 \eps)},   $$
so taking $k$ such that $k \mu_0 \geq \mez$ and using \eqref{mu0} we deduce that $ \vert J_k \vert \leq C\big( \frac{h}{\sigma}\big)^\mez.$ It follows the from Lemma \ref{estI2} that $\vert I_2 \vert \leq C\big( \frac{h}{\sigma}\big)^\mez$ and from \eqref{estI1}, \eqref{I1I2} that $\vert K_+ \vert \leq  \frac{C}{h} \big( \frac{h}{\sigma}\big)^\mez$ which completes the proof of Proposition \ref{p4.5}.
\end{proof}
\subsection{End of the proof of Theorem \ref{DISP}}
  Let us set $J_\eps = [0, h^{\mez - \eps}]$. It follows from \eqref{parametrix} and Proposition \ref{eikampl} that
 \begin{equation}\label{eqUtilde}
 \partial_{t}\widetilde{U}_h   +\frac 1 2 (W^\delta_h \partial_x + \partial_x W^\delta_h) \widetilde{U}_h 
+ ia(D_x) \widetilde{U}_h =F_h, \quad  \widetilde{U}_h\arrowvert_{t=0} =  \widetilde{U}_h (0,x),
\end{equation}
with
\begin{equation}\label{estFh}
\sup_{s \in J_\eps}\Vert F_h(s,\cdot) \Vert_{H^1(\xR)} \leq C_N h^N \Vert u_{0,h}\Vert_{L^1(\xR)}.
\end{equation}
We claim that,
\begin{equation}\label{Uzero}
 \widetilde{U}_h (0,\cdot) = u_{0,h} + v_{0,h}, \quad \Vert v_{0,h} \Vert_{H^1(\xR)} \leq C_N h^N \Vert u_{0,h}\Vert_{L^1(\xR)}.
\end{equation}
Indeed using \eqref{parametrix},\eqref{tsigma},\eqref{C3},\eqref{C17},\eqref{C18} and \eqref{C19} Êwe see that,
\begin{equation}\label{vzero}
  v_{0,h}(x) = (2 \pi h)^{-1} \iint e^{{\frac i h} (x-z)\xi} (1 - \zeta(x-z)) \chi_1(\xi) u_{0,h}(z) dz d\xi.
\end{equation}
Since on the support of $1 - \zeta(x-z)$ we have $\vert x-z \vert \geq 1$ we can integrate by parts as much as we want to obtain that for all $N \geq 1,$
$$ v_{0,h}(x) =  c_N h^{N-1}   \iint e^{{\frac i h} (x-z)\xi} [\frac{1 - \zeta(x-z)}{(x-z)^N} ]  (\partial_ {\xi}^N \chi_1)(\xi) u_{0,h}(z) dz d\xi.$$
Using the H\"older inequality we deduce  that,
$$ \vert v_{0,h} (x) \vert^2 \leq C_N h^{N-1} \Big( \int \la \frac{1 - \zeta(x-z)}{(x-z)^N}\ra^2 \vert u_{0,h}(z) \vert dz\Big) \, \Vert u_{0,h} \Vert_{L^1 (\xR)}$$
from which we deduce that,
$$ \Vert v_{0,h} \Vert_{L^2(\xR)} \leq C_N h^{N-1} \Vert u_{0,h} \Vert_{L^1(\xR)}.$$
Differentiating \eqref{vzero} with respect to $x$ and using the same trick we obtain the estimate in \eqref{Uzero}.

Now by \eqref{eqUtilde}, the  Duhamel formula and the definition in Proposition \ref{DISP} we can write,
\begin{equation}\label {S(t)=}
\begin{aligned}
&S(t,0,h)u_{0,h} = D_1 + D_2 + D_3 \quad where \quad  D_1 = \widetilde{U}_h (t,x),\\
 & D_2 = -S(t,0,h)v_{0,h}(x), \, D_3= - \int_{0}^t S(t,s,h) [F_h(s,x)]ds.
\end{aligned}
\end{equation}
First of all the estimate
\begin{equation}\label{A1bis}
\Vert D_1(t) \Vert_{L^\infty(\xR)} \leq \frac C {h^{1/4} |t|^{1/2}} \|u_{0,h}\|_{L^1(\xR)}
\end{equation}
follows from Proposition \ref{p4.5} and \eqref{1}.

Let us  estimate  $D_2$. We have by Sobolev inequality,
$$ \Vert D_2 (t) \Vert_{L^\infty (\xR)} \leq C_1 \Vert D_2(t) \Vert_{H^1(\xR)} \leq C_2 \Vert v_{0,h}\Vert_{H^1(\xR)}, $$
therefore by \eqref{Uzero},
\begin{equation}\label{A2}
\Vert D_2 (t) \Vert_{L^\infty (\xR)} \leq C_N h^N \Vert u_{0,h}\Vert_{L^1(\xR)}.
\end{equation}
Let us look now to the term $D_3$. We have,
$$ \Vert D_3 (t) \Vert_{L^\infty (\xR)} \leq C \int_{J_\eps}\Vert S(t,s,h) F_h(s,\cdot) \Vert_{H^1(\xR)} ds \leq C'\int _{J_\eps}\Vert F_h(s,\cdot) \Vert_{H^1(\xR)} ds, $$
from which we deduce,
\begin{equation}\label{A3}
\Vert D_3 (t) \Vert_{L^\infty (\xR)} \leq C_N h^N \Vert u_{0,h}\Vert_{L^1(\xR)}.
\end{equation}
 Then Theorem \ref{DISP} follows from \eqref{S(t)=}, \eqref{A1bis}, \eqref{A2}, and \eqref{A3}.

\subsection{The $TT^*$ argument}
Having proved the dispersion estimate the Strichartz estimates for the solution of \eqref{Phifquint} follow  very classically.
\begin{prop}\label{prop.strichartz}
 There exist $\eps >0$, $C>0$ such that for any $0<h<1$ and any initial data $u_{0,h}= \chi(hD_x) u_0$, we have
\begin{equation}
\label{eq.strichartz}
\|S(t,0,h) u_{0,h}\|_{L^4((0,h^{\frac 1 2 -\eps}), L^\infty({\xR}))} \leq C \|u_{0,h}\|_{H^{\frac 1 8} (\xR)}.
\end{equation} 
\end{prop}
\begin{proof}
Indeed, applying the usual $TT^*$ argument, it suffices to prove that the operator
$$ \int_0^{ h^{\frac 1 2 - \eps}} S(t,0,h) S(s,0,h)^* f(s) ds$$ 
maps continuously 
$L^{\frac 4 3}((0,h^{\frac 1 2 -\eps}),  L^1(\xR))$ to $L^4((0,h^{\frac 1 2 -\eps}),  L^\infty (\xR))$. But a direct calculation shows that since $\frac1 2 (W^\delta_h \partial_x + \partial_x W^\delta_h)$ is self adjoint, one has 
$$ S(s,0,h)^*= S(0,s,h),$$
and consequently, Proposition~\ref{prop.strichartz} follows from the classical Hardy-Littlewood-Sobolev inequality and the dispersion estimate~\eqref{eq.dispersion1}.
\end{proof}
\begin{coro}\label{SNH}
Let  $u$ be a solution of the problem
$$
\partial_t u+ \frac 1 2 (W_h^\delta \partial_x  +\partial_x W_h^\delta) u+i a(D_x) u =f,\quad u\arrowvert _{t=0}=0
$$
with $\supp \hat{f} \subset \{\frac 1 2 h^{-1}\leq \vert \xi \vert \leq 2 h^{-1}\}.$ Then we have,

$$\lA u\rA_{L^4((0,h^{\mez- \varepsilon}), L^\infty(\xR))}\le Kh^{- 1/8} \lA f\rA_{L^1((0,h^{\mez- \varepsilon}), L^2(\xR))}.$$
\end{coro}
\begin{proof}
Indeed we have,
$$ u(t, \cdot) =  \int^t_0 S(t,0,h)S^*(s,0,h)f(s, \cdot)\,ds.$$
Let us set $ J_\eps = [0, h^{\frac 1 2 - \eps}].$ It follows from Proposition \ref{prop.strichartz} that,
\begin{equation*}
\begin{aligned}
  \lA u\rA_{L^4(J_\eps, L^\infty(\xR))} &\leq C  \int_0^{ h^{\frac 1 2 - \eps}}  \Vert S(s,0,h)^* f(s, \cdot)\Vert_{H^{\frac 18}(\xR)} ds \\
  &\leq C' \int_0^{ h^{\frac 1 2 - \eps}}  \Vert  f(s, \cdot)\Vert_{H^{\frac 18}(\xR)} ds  \leq C''h^{- 1/8} \lA f\rA_{L^1(J_\eps, L^2(\xR))},
\end{aligned}
\end{equation*}
since $\hat{f}$ is supported in $\{\frac 1 2 h^{-1}\leq \vert \xi \vert \leq 2 h^{-1}\}.$
\end{proof}
\subsection{Gluying the estimates}
It remains to glue the estimates which up to now have been proved on small time intervals of size $h^{\mez- \varepsilon}.$
Recall that from Lemma~\ref{lem.3.5} we have
$$
\partial_t u_h+\frac 1 2 (W_h^\delta \partial_x  +\partial_x W_h^\delta) u_h +i\la D_x\ra^\tdm u_h =f_h \in L^\infty((0, \tau); H^{s+\varepsilon- \mez}).
$$
Let $\varphi \in C^\infty_0 ( 0,2)$, equal to $1$ on $(\mez, \tdm)$.  For $ -1 \leq k \leq h^{\varepsilon - \mez}$, define 
$$
u_{h,k}=\varphi\Bigl( \frac{t-kh^{\mez- \varepsilon}}{h^{\mez- \varepsilon}}\Bigr) u_h,
$$
which satisfies
\begin{multline}
\partial_t u_{h,k}+ \frac 1 2 (W_h^\delta \partial_x u +\partial_x W_h^\delta)u_{h,k} +i\la D_x\ra^\tdm u_{h,k} \\
=\varphi\Bigl( \frac{t-kh^{\mez- \varepsilon}}{h^{\mez- \varepsilon}}\Bigr)f_h+ h^{\varepsilon - \mez} \varphi '\Bigl( \frac{t-kh^{\mez- \varepsilon}}{h^{\mez- \varepsilon}}\Bigr) u_h ,\quad u_{h,k}\arrowvert _{t=kh^{\mez - \varepsilon}}=0,
\end{multline}
As a consequence, using Corollary \ref{SNH} we obtain, 
\begin{multline}
 \|u_{h,k}\|_{L^4((kh^{\mez -\varepsilon}, (k+2)h^{\mez -\varepsilon} ),\, L^\infty(\xR))}\\
 \leq h^{- \frac 1 8}\Bigl\|\varphi\Bigl( \frac{t-kh^{\mez- \varepsilon}}{h^{\mez- \varepsilon}}\Bigr)f_h+ h^{\varepsilon - \mez} \varphi '\Bigl( \frac{t-kh^{\mez- \varepsilon}}{h^{\mez- \varepsilon}}\Bigr) u_h\Bigr\|_{L^1((kh^{\mez -\varepsilon}, (k+2)h^{\mez -\varepsilon} ),\, L^2(\xR))}\\
 \leq  Ch^{\mez - \varepsilon- \frac 1 8}\Bigl(\|f_h\|_{L^\infty(0, \tau), \, L^2(\xR))}+ h^{\varepsilon - \mez} \| u_h\|_{L^\infty((0, \tau),\,  L^2(\xR))}\Bigr)\leq C h^{s - \frac 1 8}
\end{multline}
where in the last inequality we used the assumption $u\in L^\infty((0, \tau), H^s( \xR))$ and~\eqref{eq.estj1}.  Eventually, noticing that, 
$$ \|u_h\|^4_{L^4((0 , \tau ), \, L^\infty(\xR))}\leq \sum_{k= -1}^{h^{\varepsilon - \mez}} \|u_{h,k}\|^4_{L^4((kh^{\mez -\varepsilon}, (k+2)h^{\mez -\varepsilon} ), \, L^\infty(\xR))},
$$
we obtain 
$$ \|u_h\|^4_{L^4((0 , \tau ), \, L^\infty(\xR))}\leq C h^{\varepsilon - \mez} h^{4(s- \frac 1 8)}
$$ which implies 
$$ \|u_h\|_{L^4((0 , \tau ), \, W^{s- \frac 1 4,\infty}(\xR))}\leq C h^{\varepsilon/4}
$$
and consequently, since $h= 2^{-j}$, 
\begin{multline}
 \|u\|_{L^4((0 , \tau ), \, W^{s+\frac{\varepsilon} 8- \frac 1 4,\infty}(\xR))}\\
 \leq  \sum_{j=0}^\infty \|u_{2^{-j}}\|_{L^4((0 , \tau ), \, W^{s+\frac{\varepsilon} 8- \frac 1 4,\infty}(\xR))}
 \leq C \sum_{j=0}^\infty 2^{-j\varepsilon /4}<+\infty.
\end{multline}
\section{Classical time parametrix}
In this section we take $s> \frac{11}{2}$ and we  prove  the usual Strichartz estimates. The main step is, as before, the dispersion estimate. 
 To do so,  we seek a parametrix.  The main difference with respect to the previous section is that (in the semi-classical framework), we are looking for a large ($\mathcal{O}( h^{-1/2})$) time parametrix. As a consequence, the lower order term $T_W \partial_x$ induces oscillations. This is reflected in the fact that the new eikonal equation will be quasi-linear.
  
 We begin by an analogue of Theorem \ref{DISP}.
   
\begin{theo} \label{DISP2}
 Let $\chi \in C^\infty_0(\xR^)$ with $\supp \chi \subset \{ \xi : \frac{1}{2} \leq \vert \xi \vert \leq 2 \}$ and $t_0 \in \xR$. For any initial data $u_{0,h}= \chi(hD_x) u_0$ where $u_0 \in L^1(\xR)$ let us denote by 
 $S(t,t_0,h) u_{0,h}:= U_h $  the solution of
\begin{equation}\label{}
\partial_{t}U_h   +\frac1 2 (W^\delta_h \partial_x + \partial_x W^\delta_h) U_h  
+i a(D_x) U_h =0, \qquad U_h \mid_{t=t_0} = u_{0,h}.
\end{equation}
 Then there exists $\tau_0>0$ such that for any $0<h \leq1$ and any $|t-t_0|\leq \tau_0 $,
\begin{equation}
\label{eq.dispersion}
\|S(t,t_0,h) u_{0,h}\|_{L^\infty(\xR)} \leq \frac C {h^{1/4} |t-t_0|^{1/2}} \|u_{0,h}\|_{L^1(\xR)}.
\end{equation} 
\end{theo}

According to Lemma~\ref{lem.3.6},~Theorem \ref{DISP2}, the Duhamel formula and the same $TT^*$ argument as in Section 4,  we deduce 
\begin{coro}\label{cor.5.2}
With the notations of Lemma~\ref{lem.3.6}, we have
 \begin{equation}
\|u_h\|_{L^4((0,T), L^\infty( \xR))}=\|\Delta_j u\|_{L^4((0,T), L^\infty( \xR))} \leq C 2^{j(\frac 1 8-s)}c_j,\quad c_j \in \ell^2.
\end{equation}
\end{coro}
In the remaining of this section, we shall prove Theorem~\ref{DISP2}. We need first to refine the constructions 
in Section~\ref{sec.semiclass} to handle  large times. 
An important point in the construction of the phase function is that handling 
large times leads us to {\em non linear geometric optics}. 

Our parametrixe will be of the form \eqref{parametrix},\eqref{amplitude} that is,
\begin{equation}\label{parametrix2}
\widetilde{U}_h(t,x)= \frac 1 {2\pi h} \iint e^{\frac i h (\Phi(t,x,\xi,h)- z\xi)} \widetilde {B}(t,x,z,\xi,h) u_{0,h}(z) dz d\xi,
\end{equation}
where $\Phi$  will satisfy the eikonal equation and
\begin{equation}\label{amplitude2}
 \widetilde {B}(t,x,z,\xi,h) = B(t,x,\xi,h) \zeta(x-z-th^{- \mez}a'(\xi)),
 \end{equation}
where $B$ will satisfy the transport equations and $ \zeta \in C_0^\infty (\xR)$, $\zeta(s) = 1$ if $\vert s \vert \leq 1$, $\zeta (s) = 0$ if $\vert s \vert \geq 2.$

  \subsection{Notations}
  In this section we fix
$$s> \frac{11}{2} \text{ and } \delta = \frac{1}{s- \frac{3}{2}}<\frac 1 4. $$

 As before  we shall set $2^{-j} =h$, where  $j\in \xN$ and we shall work with the semiclassical time $\s = th^{-\mez}$.

 In addition to the function $\chi$ introduced in Theorem~\ref{DISP2}, we shall use two more cut-off functions $\chi_j \in C_0^\infty (\xR), j = 1,2$,  such that,
 \begin{equation}\label{chi0chi1}
 \left\{
 \begin{aligned}
  &\supp \chi_1 \subset \{ \xi : \frac{1}{3} \leq \vert \xi \vert \leq 3 \}, \quad \chi_1 = 1\,  \text{on the support of }  \chi, \\
  &\supp \chi_0 \subset \{ \xi : \frac{1}{4} \leq \vert \xi \vert \leq 4 \}, \quad\chi_0 = 1 \,\text{on the support of }  \chi_1.
  \end{aligned}
  \right.
\end{equation}
 Recall  that we have
\begin{align*}
\tag{i} &
W\in L^\infty ([0,T], W^{2,\infty}(\xR)), \partial_tW \in L^\infty([0,T], W^{1,\infty}(\xR)) 
 \quad (\text{Lemma \protect{\ref{defW}}}),\\ 
 \tag{ii}&W_h^\delta= S_{[\delta (j-3)]}(W)
\text{ satisfies }\lA \partial^\alpha_x W_h^\delta \rA_{L^\infty_{t,x}} 
\leq C_\alpha \lA \partial^\alpha_x W \rA_{L^\infty_{t,x}},\\
\tag{iii} & a(\xi)=\chi_{0}(\xi)\la\xi\ra^{\tdm},
\end{align*}
\begin{defi}
For  small $h_0,\tau_0$  to be fixed, we introduce the sets
 \begin{equation*}
 \begin{aligned}
 \Omega &=\left\{(t,x,\xi,h)\in \xR^4\, :\, h\in (0, h_0), |t| < \tau_0, 1<\la \xi\ra < 3\right\}\\
 \mathcal{O}&=\left\{(\s,x,\xi,h)\in \xR^4\, :\, h\in (0, h_0), |\s| <  \tau_0 h^{-\mez}, 1<\la \xi \ra < 3 \right\}.
\end{aligned}
\end{equation*}
If
$m\in \xR $ and $\varrho\in \mathbb{R}^+$,  we denote by $S^m_{\varrho}(\Omega)$ (resp.$ S^m_{\varrho}(  \mathcal{O})$) the set of all 
functions $f$ on $  \Omega$ (resp.$\mathcal{O}$) which are $C^\infty$ with respect to 
$(t,x,\xi)$ (resp.$(\s,x,\xi)$ and satisfy the estimate
\begin{equation}
\la \partial_x^\alpha f(t,x,\xi,h)\ra (resp.
\la \partial_x^\alpha f(\s,x,\xi,h)\ra )
\le C_{\alpha}  h^{m-\varrho\alpha}, 
\end{equation}
for all $(t,x,\xi,h)\in \Omega$ (resp.$(\s,x,\xi,h)\in \mathcal{O})$.
\end{defi}
\subsection{The eikonal and transport equations}
In all this section we keep the notations of \eqref{tsigma},\eqref{C2}, and \eqref{unpeudeL} to \eqref{Anouveau}.

The main result  is the following.
\begin{prop}\label{eikampl2}
There exist a phase $\Phi$ of the form 
$$\Phi(t,x,\xi,h)= x \xi - h^{-\mez}ta(\xi)+h^{\mez}\Psi(t,x,\xi,h)$$
 with $\partial_x \Psi \in S^0_{\delta}(\Omega)$ and an amplitude $B  \in  S^{0}_\delta  \Omega)$ such that, with $\widetilde B$ defined in \eqref{amplitude2},
 \begin{equation}
 \mathcal{L}_0 \left( e^{\frac{i}{h}\Phi}\widetilde{B}  \right)= e^{\frac{i}{h}\Phi}R_h.
\end{equation}
    and for all $N\in \xN$ we have,
\begin{equation}\label{estReste2}
\Big\Vert \iint e^{\frac{i}{h}(\Phi(t,x,\xi,h) - z\xi)}R_h(t,x,z,\xi,h)u_{0,h}(z)\,dz\, d\xi \Big\Vert_{ 
H^1(\xR_x)} \leq C_Nh^N \Vert u_{0,h} \Vert_{L^1(\xR)},
\end{equation}
for all t in $[0, \tau_0]$.
 \end{prop}
 \begin{proof}
   According to \eqref {Anouveau} we have,
\begin{equation*}
\left\{
\begin{aligned}
A_0&= a(x,\xi +h^\mez\partial_x \psi)b,\\
A_1&=\frac{h}{i}a(x,\xi +h^\mez \partial_x \psi) \partial_x b +\frac{h^{\frac 3 2}}{i}(\partial_x^2 \psi)a'(x,\xi +h^\mez\partial_x \psi)b.
\end{aligned}
\right.
\end{equation*}
We  deduce  that $\mathcal{L}_0(e^{{\frac i h} \varphi} \widetilde{b}) = e^{{\frac i h} \varphi} r$ with
\begin{equation}\label{r=bis}
\begin{aligned}
r &= i \left\{ h^\mez\partial_\s \psi - a(\xi) + a(x,\xi +h^\mez \partial_x \psi) + h W^\delta_h \partial_x \psi + h^\mez \xi W^\delta_h \right\}b\,\zeta \\
 &+h \Big \{ \partial_ \s b + a(x,\xi +h^\mez \partial_x \psi) \partial_x b+ h^ \mez W^\delta _h \partial_x b  + \frac{1}{2}h^ \mez (\partial_x W^\delta_h) b \\
 &+ h^\mez (\partial_x^2 \psi)a'(x,\xi +h^\mez \partial_x \psi)b + \frac {i}{h} \sum _{k=2}^{M-1} A_k \Big \} \zeta +i \sum_{j=1}^4 r_j,
 \end{aligned}
\end{equation}
where $r_1,r_2 $ are defined in \eqref{C9}, \eqref{C10}, $r_3$ in \eqref{Anouveau} and
\begin{equation}\label{r4bis}
r_4= \frac{h}{ i} \left\{-a'(\xi)+h^\mez W^\delta_h \right\}b\, \zeta'.
\end{equation}

\subsubsection{The eikonal equation}
As already mentioned, an 
important point in the construction of the phase function is that handling 
large times leads us to {\em non linear geometric optics}. 
Namely, we determine $\psi$ by solving the following nonlinear problem,
\begin{equation}\label{eikonal}
\left\{
\begin{aligned}
&\partial_\s \psi +\frac{a(\xi+{h^\mez}\partial_x \psi)-a(\xi)}{h^\mez}   + {h^\mez} W_h^\delta({h^\mez} \s,x) \partial_x\psi = - W_h^\delta({h^\mez} \s,x)\xi,\\
&\psi(0,x,\xi,h)=0.
\end{aligned}
\right.
\end{equation}
In this system, $\xi$ and $h$ are seen as parameters. 
We begin by establishing that the solutions exist for a time interval of size $h^{-\mez}$ and 
satisfy some uniform estimates.
\begin{prop}\label{eiko}
There exists $\tau_0 >0$ such that the problem \eqref{eikonal} has a 
unique $C^\infty$ solution $\psi$ in the set $ \mathcal{O}$ 
such that $\partial_x \psi$ and $\partial^2_x \psi$ are uniformly 
bounded  on  $ \mathcal{O} $  by 
$ C\big( \Vert (\eta, \psi) \Vert_{L^\infty (I, H^{s+ \mez}(\xR) \times H^{s}(\xR))}\big)$ 
where $C$ is an increasing function from $\xR^+$ to itself.
\end{prop}
\begin{proof}
Let us differentiate the equation \eqref{eikonal} with respect to $x$ and let us set $\psi_1= \partial_x \psi.$ Then  $\psi_1$ is solution of the quasi-linear equation
\begin{equation}\label{psi1}
\partial_\s \psi_1 + A(\s,x,h,\xi,\psi_1)\partial_x \psi_1 = B(\s,x,h,\xi,\psi_1),\quad \psi_1(0,x,h,\xi)=0.
\end{equation}
where 
\begin{equation}\label{A,B}
\left\{
\begin{aligned}
&A(\s,x,h,\xi,z)= a'(\xi + h^\mez z) +h^\mez W_h^\delta( h^\mez\s ,x),\\
&B(\s,x,h,\xi,z)=-h^\mez (\partial_xW_h^\delta)( h^\mez\s ,x)z - \xi(\partial_xW_h^\delta) ( h^\mez\s ,x).
\end{aligned}
\right.
\end{equation}
We shall solve \eqref{psi1} by the method of characteristics.\\
The characteristics are given by the system
\begin{equation}\label{car}
\left\{
\begin{aligned}
&\dot{\s}(s)=1,\quad \s(0)=0,\\
&\dot{X}(s)= A(s,X(s),h,\xi,Z(s)), \quad X(0)=x,\\
&\dot{Z}(s)=B(s,X(s),h,\xi,Z(s)), \quad Z(0)=0.
\end{aligned}
\right.
\end{equation}
Since $A$ is uniformly bounded and $\vert B \vert \leq C_1 +C_2 \vert z\vert,$ the 
above system has a unique  global solution (i.e defined for $ s \in [0, + \infty[$).

\subsubsection{Properties of the flow.}
$(i)$ We have,
\begin{equation}\label{Xpoint}
 \exists \tau_0>0,c_1>0,c_2>0 :   c_1 \leq \vert \dot{X}(s)\vert \leq c_2, \quad  0 \leq s \leq \tau_0 h^{-\mez}.
\end{equation}
We show first that
\begin{equation}\label{Z1}
h^{\mez} \vert Z(s) \vert \leq C\tau_0\Vert \partial_x W \Vert_{L^\infty} \exp\big( \tau_0 \Vert \partial_x W \Vert_{L^\infty}\big),\quad 0 \leq s \leq \tau_0 h^{-\mez}.
\end{equation}
To see this we integrate the equation satisfied by $Z$ and use \eqref{A,B}. We obtain
\begin{equation}
\vert Z(s) \vert \leq C \Vert \partial_xW\Vert_{L^\infty_{t,x}} \vert s \vert 
+ h^\mez \Vert \partial_xW \Vert_{L^\infty_{t,x}} \int_0^s \vert Z(\s) \vert d\s, \quad 0 
\leq s \leq \tau_0 h^{-\mez}.
\end{equation}
Then \eqref{Z1} follows from the Gronwall inequality.

On the other hand, setting $m(s) = (s,X(s),h,\xi,Z(s)),$ we have
\begin{align*} 
A(m(s)) &= a'(\xi) + h^ \mez Z(s) \int_0^1a''(\xi + \lambda h^\mez Z(s))d \lambda 
+ h^\mez W_h^\delta( h^\mez \s ,X(s))d\s\\
& := a'(\xi) +R
\end{align*}
where 
\begin{equation*}
\vert R \vert \leq \tau_0 C(\tau_0, \Vert \partial_x W \Vert_{L^\infty_{t,x}}) \Vert a'' \Vert_{L^\infty} 
+ h^\mez \Vert W \Vert_{L^\infty_{t,x}})
\end{equation*}
Since for $1/2 \leq \vert \xi \vert \leq 3$ we have $ \vert a'(\xi)\vert \geq 2c_1>0$ 
we obtain 
$$
\vert A(s,X(s),h,\xi,Z(s)) \vert \geq c_1
$$ when $0 \leq s \leq \tau_0 h^{-\mez}$, $(\tau_0$ and $h$  small enough). This proves \eqref{Xpoint}.

$(ii)$ We have,
\begin{equation}\label{Xpp}
\vert \ddot{X}(s) \vert \leq h^\mez C \Big(\Vert \partial_x W \Vert_{L^\infty_{t,x}}+\Vert \partial_t W \Vert_{L^\infty_{t,x}} \Big),\quad 0 \leq s \leq \tau_0 h^{-\mez}.
\end{equation}
Indeed let us set $m(s) = (s,X(s),h,\xi,Z(s)).$Then we have,
\begin{equation*}
\ddot{X}(s) =(\partial_s A)(m(s)) + (\partial_x A)(m(s)) \dot{X}(s) +( \partial_z A)(m(s))\dot{Z}(s).
\end{equation*}
Moreover we have,
\begin{align*}
&(\partial_s A)(m(s))= h(\partial_s W^\delta_h)(\s h^\mez,x), 
\quad (\partial_x A)(m(s))= h^\mez(\partial_x W^\delta_h)(\s h^\mez,x)\\
&  (\partial_z A)(m(s))= a''(\xi + h^\mez Z(s)).
\end{align*}
 Then \eqref{Xpp} follows from the expressions of 
 $\dot{X}(s), \dot{Z}(s)$ and \eqref{Z1}.
 
 $(iii)$ We improve now \eqref{Z1}. We have,
 \begin{equation}\label{Z2}
 \vert Z(s) \leq  C\Big(\sum_{\vert \alpha \vert \leq 1} \Vert \partial^\alpha W \Vert_{L^\infty_{t,x}}\Big), \quad \partial =(\partial_t, \partial_x).
 \end{equation}
Indeed we can write
\begin{equation}\label{Z3}
Z(s) = - \xi \int_0^s(\partial_xW_h^\delta)(\s h^\mez, X(\s)) d\s - h^\mez \int_0^s(\partial_xW_h^\delta)(\s h^\mez, X(\s))Z(\s) d\s
  \end{equation} 
Now, using \eqref{Xpoint} we have,
\begin{equation}\label{astuce}
(\partial_x W_h^\delta)(\s h^\mez, X(\s)) = 
\frac{\partial_\s [W_h^\delta(\s h^\mez, X(\s))]}{\dot{X}(\s)} 
-  \frac{h^\mez (\partial_s W_h^\delta)(\s h^\mez, X(\s))}{\dot{X}(\s)}.
\end{equation}
After an integration by parts we obtain,
\begin{align*}
I &=: \int^s_0 (\partial_x W_h^\delta)(\s h^\mez, X(\s)) d\s 
= \frac{W_h^\delta(sh^\mez,X(s))}{\dot{X}(s)} -  \frac{W_h^\delta(0,X(0))}{\dot{X}(0)}\\
&+ \int_0^s \frac{\ddot {X}(\s)}{(\dot {X}(\s))^2} W_h^\delta(\s h^\mez, X(\s)) d\s 
- h^\mez \int_0^s \frac{1}{\dot {X}(\s)} (\partial_sW_h^\delta)(\s h^\mez, X(\s)) d\s.
\end{align*}
Using \eqref{Xpoint},\eqref{Xpp} we deduce that for $ 0 \leq s \leq \tau_0 h^{-\mez}$ we have,
\begin{equation*}
\vert I \vert \leq C\Big(\sum_{\vert \alpha \vert \leq 1} 
\Vert \partial^\alpha W \Vert_{L^\infty_{t,x}}\Big) \quad \partial =(\partial_s,\partial_x).
\end{equation*}
It follows  from \eqref{Z3} that,
\begin{equation*}
\vert Z(s) \vert \leq C\Big(\sum_{\vert \alpha \vert \leq 1} 
\Vert \partial^\alpha W \Vert_{L^\infty_{t,x}}\Big) + 
\Vert \partial_x W_h^\delta \Vert_{L^\infty_{t,x}}h^\mez \int_0^s \vert Z(\s) d\s,
\end{equation*}
which using Gronwall inequality proves \eqref{Z2}.

We are going now to give some estimates on the $x$-derivative of the flow.

We claim that,
\begin{equation}\label{der1X}
\vert \frac{\partial X}{\partial x} (s) \vert +\vert \frac{\partial Z}{\partial x} (s) \vert \leq 
C\Big(\lA \partial_t W\rA_{W^{1,\infty}}+\lA W\rA_{W^{2,\infty}}
\Big) \quad 0 \leq s \leq \tau_0 h^{-\mez}
\end{equation}
\begin{equation}\label{der2X}
\vert \frac{\partial X}{\partial x} (s) - 1 \vert \leq \frac{1}{2}, \quad 0 \leq s \leq \tau_0 h^{-\mez},
\end{equation}
if $ \tau_0 C\Big( \lA \partial_t W\rA_{W^{1,\infty}}+\lA W\rA_{W^{2,\infty}})\Big) $ is small enough.

Indeed using \eqref{car} we can write,
\begin{align*}
&\dot{ \frac{\partial X}{\partial x}} (s) = 
h^\mez (\partial_x W_h^\delta)(s h^\mez, X(s))\frac{\partial X}{\partial x}(s)+
h^\mez a''(\xi + h^\mez Z(s)) \frac{\partial Z}{\partial x}(s), \\
&\dot{\frac{\partial Z}{\partial x}}(s)= -\xi(\partial^2_x W_h^\delta)(s h^\mez, X(s))\frac{\partial X}{\partial x}(s)
-h^\mez (\partial^2_x W_h^\delta)(s h^\mez, X(s))\frac{\partial X}{\partial x} (s) Z(s) \\
&-h^\mez (\partial_x W_h^\delta)(s h^\mez, X(s))\frac{\partial Z}{\partial x}.
\end{align*}
From the first equation we deduce
\begin{equation}\label{der3X}
\vert \frac{\partial X}{\partial x} (s) \vert \leq 1+ h^\mez \Vert (\partial_x W_h^\delta) 
\Vert_{L^\infty_{t,x}} \int _{0}^s \vert \frac{\partial X}{\partial x} (\s) \vert d\s 
+ h^\mez \Vert a'' \Vert_{L^\infty} \int_{0}^s \vert \frac{\partial Z}{\partial x} (\s) \vert d\s.
\end{equation}
From the second equation we deduce,
\begin{equation}\label{derZ}
\left\{
\begin{aligned}
&\frac{\partial Z}{\partial x} (s) = I_1=I_2+I_3 \quad \text{where},\\
& I_1=- \xi \int^s_0 (\partial^2_x W_h^\delta)(\s h^\mez, X(\s)) \frac{\partial X}{\partial x} (\s)d\s, \\
&I_2= -h^\mez \int^s_0 (\partial^2_x W_h^\delta)(\s h^\mez, X(\s)) \frac{\partial X}{\partial x} Z(\s)d\s, \\
&I_3= -h^\mez \int^s_0 (\partial_x W_h^\delta)(\s h^\mez, X(\s)) \frac{\partial Z}{\partial x}d\s .
\end{aligned}
\right.
\end{equation}
We have easily,
\begin{equation}
\vert I_3 \vert \leq h^\mez \Vert \partial_x W\Vert_{L^\infty_{t,x}} 
\int_0^s \vert \frac{\partial Z}{\partial x}(\s) \vert d\s
\end{equation}
Moreover using \eqref{Z2} we get,
\begin{equation}
\vert I_2 \vert \leq h^\mez C\Big(\sum_{\vert \alpha \vert \leq 1} 
\Vert \partial^\alpha W \Vert_{L^\infty_{t,x}}\Big) \Vert\partial_x^2 W 
 \Vert_{L^\infty_{t,x}}  \int_0^s \vert \frac{\partial X}{\partial x}(\s) \vert d\s.
\end{equation}
We are left with the estimate of $I_1$. We use \eqref{astuce} applied to 
$\partial_x W_h^\delta.$We obtain
\begin{align}
I_1=  &-\xi \int_0^s
\frac{\partial_\s [\partial_xW_h^\delta(\s h^\mez, X(\s))]}{\dot{X}(\s)} \frac{\partial X}{\partial x}(\s)d\s \\
&- Êh^\mez \xi \int_0^s 
\frac{ (\partial_x\partial_s W_h^\delta)(\s h^\mez, X(\s))}{\dot{X}(\s)}\frac{\partial X}{\partial x}(\s)d\s =: A + B.
\end{align}
We see easily that
\begin{equation}
\vert B \vert \leq Ch^\mez \Vert \partial_s \partial_x W 
\Vert_{L^\infty_{t,x}} \int_0^s \vert \frac{\partial X}{\partial x}(\s) \vert d\s.
\end{equation}
Let us consider the term $A$. After an integration by parts one can write,
\begin{align*}
A &= - \xi \big( 
\frac{ (\partial_x W_h^\delta)(s h^\mez, X(s))}{\dot{X}(s)}\frac{\partial X}{\partial x}(s) 
-\frac{ (\partial_x W_h^\delta)( 0, x)}{\dot{X}(0)}\\
&\quad+\xi\int_0^s (\partial_xW_h^\delta)(\s h^\mez, X(\s))
\dot{ \frac{\partial X}{\partial x}} (s)\frac{d \s}{\dot{X}(\s)} \\ 
&\quad-  \xi \int_0^s (\partial_xW_h^\delta)(\s h^\mez, X(\s))
\frac{\ddot{X}(\s)}{(\dot{X}(\s))^2}\frac{\partial X}{\partial x}(\s)d\s.
\end{align*}
Using \eqref{der2X}, \eqref{Xpoint}, \eqref{Xpp} and the equation 
satisfied by $\frac{\partial X}{\partial x}$ we obtain
\begin{equation}\label{der4X}
\vert A\vert \leq C\Big(\Vert \partial_x W \Vert_{L^\infty_{t,x}} \Big) 
\Big(1+ h^\mez \int_0^s (\vert \frac{\partial X}{\partial x}(\s) \vert 
+ \vert \frac{\partial Z}{\partial x}(\s) \vert) d\s\Big).
\end{equation}
Using \eqref{derZ} to \eqref{der4X} we obtain
\begin{align*}
\vert \frac{\partial Z}{\partial x} (s)\vert &\leq  
C\big( \Vert \partial_x W  \Vert_{L^\infty_{t,x}}\big) \\
&\quad+     
C\left( \lA \partial_t W\rA_{L^\infty_{t}W^{1,\infty}_x}+\lA W\rA_{L^\infty_{t}W^{2,\infty}_x}\right) 
\int_0^s 
\Big(\la \frac{\partial X}{\partial x} (\s) \ra +\la \frac{\partial Z}{\partial x} (\s)\ra \Big) d\s
\end{align*}
so using \eqref{der3X} and the Gronwall inequality we obtain \eqref{der1X}.

Then coming back to the equation satisfied by  $\frac{\partial X}{\partial x}$ we deduce
\begin{equation*}
\la \frac{\partial X}{\partial x}(s) - 1 \ra \leq 
C\left( \lA \partial_t W\rA_{L^\infty_{t}W^{1,\infty}_x}+\lA W\rA_{L^\infty_{t}W^{2,\infty}_x}\right) h^\mez \vert s \vert
\end{equation*}
for $ 0 \leq s \leq \tau_0 h^{-\mez}.$ Therefore taking $\tau_0$ small enough 
we obtain \eqref{der2X}.
\subsubsection{Resolution of the eikonal equation.}
We claim now that for all  $s$ in  $[0, \tau_0 h^{-\mez}]$ the map 
$x \to X(s,x$) is a global diffeomorphism from $\xR$ to $\xR.$ 
Indeed we have for such fixed $s$,
\begin{equation*}
X(s,x) = x + \int_0^s A(\s,X(\s),h,\xi,Z(s))d\s.
\end{equation*}
Since $A$ is bounded by $\Vert a' \Vert_{L^\infty} 
+ h^\mez \Vert W\Vert_{L^\infty_{t,x}} $ we have 
$ \lim_{\vert x\vert \to + \infty} \vert X(s,x) \vert = + \infty.$ 
Moreover by \eqref{der2X} we have, $\frac{\partial X}{\partial x}(s,x) \neq 0$ 
for all $ 0 \leq s \leq \tau_0 h^{-\mez}$ and all $x\in \xR$. Therefore our claim 
follows from a well known result by Hadamard. Then
\begin{equation}\label{XY}
X(s,x) = y \Leftrightarrow x = Y(s,x), \quad x,y \in \xR,
\end{equation}
and the function $(s,y) \to Y(s,y)$ is $C^\infty$ by the implicit function theorem. 
Let us consider then the set
$$ S = \{(s,X(s,x),Z(s,x), \,  0 \leq s \leq \tau_0 h^{-\mez}, x \in \xR\}.$$
It follows from \eqref{XY} that $S$ is a submanifold of $\xR^3$ of dimension 
two to which the vector field  $ L = \frac{\partial}{\partial \s} 
+ A(\s,x,h,\xi,z)\frac{\partial}{\partial x} + B(\s,x,h,\xi,z)\frac{\partial}{\partial z}$ 
is tangent. It follows then from \eqref{XY}  that the function $ \psi_1(s,y,h,\xi) = 
Z(s, Y(s,y))$ is the solution of our eikonal equation \eqref{psi1}. Then $\psi_1 \in L^\infty.$ 
Moreover we have 
$$ \frac{\partial \psi_1}{\partial y} = \frac{\partial Z}{\partial x}(s, Y(s,y))
\frac{\partial Y}{\partial y}(s,y)= \frac{\partial Z}{\partial x}(s, Y(s,y))(\frac{\partial X}{\partial x}(s,Y(s,y)))^{-1},
$$
so, since $\frac{\partial Z}{\partial x}$ is bounded and from \eqref{der2X} we deduce, 
$$ \vert \frac{\partial \psi_1}{\partial y}\vert \leq 
C\Big( \lA \partial_t W\rA_{W^{1,\infty}}+\lA W\rA_{W^{2,\infty}}\Big).$$
It follows that the solution $\psi$ of our eikonal equation \eqref{eikonal} is such that
$$\frac{\partial \psi}{\partial x}\in L^\infty_{t,x}\,,
\quad \frac{\partial^2 \psi}{\partial x^2}\in L^\infty_{t,x} \, ,
$$
uniformly in $h,\xi.$
\end{proof}

\subsubsection{Properties of the solution.}
We investigate in this section futher regularity of the solution $\psi.$
\begin{prop}\label{eikS_0}
 Let $\psi$ be the  solution of 
$\eqref{eikonal}$ given by Proposition $ \ref{eiko}$ Then we have
$\psi \in L^\infty (\mathcal{O}), \partial_x \psi \in L^\infty (\mathcal{O}), 
\partial^2_x \psi \in S^0_\delta (\mathcal{O}).$
\end{prop}
\begin{proof}
The first two claims have been proved in Proposition $\ref{eiko}$, 
let us consider the third one. We shall prove that $\partial_x \psi_1 \in S^0_\delta (\mathcal{O})$ 
where $\psi_1 = \partial_x \psi.$ Let us set for $\frac{1}{2} \leq \vert \xi \vert \leq 2,$
\begin{equation}
v(\s, x, \xi, h) = \partial_x \psi_1 (\s, x, \xi, h) - \frac {\xi}{a'(\xi)} \partial_x W^\delta_h (\s h^\mez,x).
\end{equation}
Then according to $\eqref{psi1}$ the function $v$ is solution of the equation,
\begin{equation}\label{eqv}
\begin{aligned}
\partial_\s v + (a'(\xi + h^\mez \psi_1) +& h^\mez W_h^\delta) \partial_x v 
+h^\mez a''(\xi + h^\mez \psi_1) v^2\\ +&(2h^\mez a''(\xi + h^\mez \psi_1) 
\frac {\xi}{a'(\xi)}\partial_x W_h^\delta + 2 h^\mez \partial_x W_h^\delta) v = f,
\end{aligned}
\end{equation}
where,
\begin{equation}
\begin{aligned}
f = &-  \frac {\xi}{a'(\xi)}h^\mez \partial_x \partial_t W_h^\delta + \frac {\xi}{a'(\xi)} 
\partial_x^2 W_h^\delta \big(a'(\xi + h^\mez \psi_1) -a'( \xi)\big ) \\
 &- h^\mez  \frac {\xi}{a'(\xi)} W_h^\delta  \partial_x^2 W_h^\delta
+ h^\mez ( \frac {\xi}{a'(\xi)})^2 a''(\xi + h^\mez \psi_1) (\partial_x W_h ^\delta)^2 \\
&+ 2 h^\mez\frac {\xi}{a'(\xi)} ( \partial_x W_h^\delta)^2 - h^\mez (\partial_x^2 W_h^\delta) \psi_1.
\end{aligned} 
\end{equation}
Let us set $ \Lambda = h^\delta \partial_x.$We shall prove by 
induction on $k \in \xN$ that,
\begin{equation}\label{rec}
\Lambda ^j v \in L^\infty (\mathcal{O}) , \quad 0 \leq j \leq k.
\end{equation}
This will imply our claim since $\partial_x W_h ^\delta \in S^0_\delta (\mathcal{O}).$
 
Notice that \eqref{rec} is true for $k = 0$ by Proposition $\ref{eiko}$. Assume 
it is true up to the order $k -1$. 
It follows then, using the Faa-di Bruno formula that,
\begin{equation}\label{FDB}
\Lambda^l [b(h^\mez \psi_1)] \in h^{\mez + \delta}L^\infty (\mathcal{O}), \quad 1 \leq l \leq k,
\end{equation}
for any $C^\infty$-bounded function $b$ from $\xR$ to $\xR$.
  
Applying the operator $\Lambda^k$ to both sides of $\eqref{eqv}$ and using $\eqref{FDB}$ 
and the fact that $\partial^\alpha W_h^\delta \in S^0_\delta (\mathcal{O})$ for 
$\partial = (\partial_t,\partial_x), \vert \alpha \vert \leq 2, \alpha \neq (2,0)$ 
we find that $\Lambda^k v $ is solution of the problem
\begin{equation}
(\partial_\s  + a'(\xi+h^\mez \psi_1)\partial_x  +  h^\mez W_h^\delta \partial_x 
+ h^\mez d(\s, x , \xi, h)) \Lambda^k v \in h^\mez L^\infty(\mathcal{O}),
\end{equation}
where $d \in L^\infty(\mathcal{O}_\eps)$. 
  
Let us set $\tilde{v}_k = (\Lambda^kv)(\s, x+ a'(\xi)\s).$ Then $\tilde{v}_k$ is solution of the problem,
$$  \Big (\partial_\s  + h^\mez 
\int_0^1 a''(\xi+\lambda h^\mez \tilde{\psi_1}) \,d\lambda \,\tilde{\psi_1} \partial_x 
+h^\mez \tilde{W}_h^\delta \partial_x + h^\mez \tilde{d}(\s, x , \xi, h)\Big)
\tilde{v}_k \in h^\mez L^\infty(\mathcal{O}),
$$

Then the desired conclusion follows from the following Lemma.
\renewcommand{\qedsymbol}{}
\end{proof}
\begin{lemm}\label{lb}
  Let $c_1,c_2$ be two 
real valued functions such that 
$c_1, \partial_x c_1, c_2 $ belong to $L^\infty (\mathcal{O})$ 
and $P=\partial_\s +h^\mez c_1(\s,x,\xi,h)\partial_x +h^\mez c_2(\s,x,\xi,h)$. 
Then for any $F \in L^\infty(\mathcal{O})$, the problem 
$$
Pu = F,\quad u\arrowvert_{\s=0}=0,
$$
has a unique solution $u$ which satisfies the estimate
\begin{equation*}
\vert u(\s,x,\xi,h) \vert \leq C \int_0^\s  \Vert F(s,\cdot,\xi,h) \Vert_{L^\infty (\xR)}\, ds,
\end{equation*}
for all $(\s,x,\xi,h)$ in $\mathcal{O}$, uniformly in $h$.
\end{lemm}
\begin{proof}[Proof of Lemma~\ref{lb}] Let us set 
$t=\s h^\mez$ and $\widetilde{c}_j(t,x)=c_j(h^{-\mez}t,x)$, $j=1,2$. 
Then we are led to the problem
$$
\widetilde P \widetilde u =h^\mez \widetilde F,\quad \widetilde u\arrowvert_{\s=0}=0 \qquad 
(t\in [0,\tau_0],x\in\xR),
$$
where $\widetilde P=\partial_t + \widetilde c_1(t,x)\partial_x + \widetilde c_2(t,x)$. Recall that 
$\widetilde c_1\in L^\infty$, $\partial_x \widetilde c_1 \in L^\infty$, $\widetilde c_2\in L^\infty$. 
Then the claim of the lemma follows from the classical method of characteristics. 
Indeed, the characteristics are given by 
$t(s)=s$ and
$$
\dot{X}(s,x)=\widetilde c_1(s,X(s,x)),\quad X(0,x)=x.
$$
Then $x\mapsto X(t,x)$ is globally invertible for each $t\in [0,\tau_0]$ i.e. $X(t,x)= y \Leftrightarrow x = Y(t,y)$ with $Y \in C^0\cap L^\infty$. Then
$$
\frac{d}{dt}\left[\widetilde{u}(t,X(t,x))\right]=\widetilde{c}_2(t,X(t,x))u(t,X(t,x))+F(t,X(t,x)).
$$
Therefore $\widetilde{u}$ given by
$$
\widetilde{u}(t,y)=\exp\left(-\int_0^t \widetilde{c}_2(t',X(t',Y(t,y))\, dt' \right)
\int_0^t F(t',X(t',Y(t,y))\, dt'
$$
is our solution.
\end{proof}
\begin{coro}\label{LU=F}
Let $\psi$ be defined by Proposition \ref{eiko} and $L$ be the operator
\begin{equation}
L = \partial_\s + a'(\xi + h^\mez \partial_x \psi) \partial_x + h^\mez d_1\partial_x + h^\mez d_2,
\end{equation}
 where $d_1, \partial_x d_1,d_2$ are real valued and belong to  $S^0_\delta (\mathcal{O})$. \\
 Then for any  $F$  such that $ \Vert \Lambda^j F  \Vert_{L^\infty (\mathcal{O})}$ is finite for every $ j \in \xN$,  the problem
$$Lu =F, \quad u\arrowvert_{\s = 0} = 0,$$
 has a unique solution which satisfies the estimate,
 \begin{equation}\label{lambdak}
 \vert \Lambda^k u(\s,x,\xi,h) \vert \leq C_k \s  \sum_{j = 0}^k \Vert \Lambda^j F  \Vert_{L^\infty (\mathcal{O})}, 
  \end{equation}
  for all $(\s,x,\xi,h) \in \mathcal{O}$, where $\Lambda = h^\delta \partial_x$.
\end{coro}
\begin{proof}
Since by Proposition \ref{eikS_0} we have $\partial^2_x \psi \in S^0_\delta (\mathcal{O})$ one can write,
$$L = \partial_\s + a'(\xi)\partial_x + h^\mez d_3\partial_x \psi   + h^\mez d_2.$$
   where $d_3,\partial_x d_3, d_2$ belong to  $S^0_\delta (\mathcal{O})$. Setting 
   \begin{align*}
   U &= u(\s,x+\s a'(\xi),\xi,h), \quad c_1 = d_3(\s,x+\s a'(\xi),\xi,h) , \\
   c_2 &=d_2(\s,x+\s a'(\xi),\xi,h), 
   \end{align*}
   we see that $c_1,\partial_x c_1, c_2$ belong to $ S^0_\delta (\mathcal{O})$ and $U$ is a solution of the equation
   \begin{equation}
   L_1 U := (\partial_\s + h^\mez c_1\partial_x \psi   + h^\mez c_2)U = F_1.
   \end{equation}
   We shall prove by induction on $k$ that $U$ satisfies the estimate \eqref{lambdak}. The case $k = 0$ follows immediately from Lemma \ref{lb}. Assume now that \eqref{lambdak} is true up to the order $k -1, k \geq 1$. 
By the Leibniz formula we have
\begin{align*}
 L\Lambda^k U +k h^\mez (\partial_x c_1) \Lambda^k  U
 &=- h^\mez \sum_{i=0}^{(k-2)^+}\begin{pmatrix} k \\ i \end{pmatrix}(\Lambda^{k-i}c_1) \Lambda^i \partial_x U \\
 &-  h^\mez \sum_{i=0}^{k-1}\begin{pmatrix} k \\ i \end{pmatrix}(\Lambda^{k-i}c_2) \Lambda^i  U
 +  \Lambda^k F_1 := G_k.
\end{align*}
 The sum in the first line can be written,
 $ - h^\mez \sum_{i=1}^{k-1}\begin{pmatrix} k \\ i-1 \end{pmatrix}(\Lambda^{k-i} \partial_x c_1) \Lambda^i U. $
According to our assumptions on $c_1,c_2$,we can apply the Lemma \ref{lb} to the operator $ L + k h^\mez (\partial_x c_1)$.We obtain, using the induction and the fact that  $h^\mez \s \leq \tau_0$,
$$ \vert \Lambda^k U(\s,x,\xi,h) \vert \leq \s \Vert \Lambda^k F \Vert_{L^\infty(\mathcal{O})}  + C \s \sum_{j = 0}^{k -1} \Vert \Lambda^j F \Vert_{L^\infty(\mathcal{O})}, $$
which completes the induction.
 \end{proof}
\begin{prop}\label{psixi}
 Let $\Lambda = h^\delta \partial_x$. The solution of $\eqref{eikonal}$ given by Proposition $\ref{eiko}$ satisfies the estimates,
\begin{equation}\label{psixi1}
\begin{aligned}
&\vert \Lambda^k \partial_\xi \psi (\s,x,\xi,h)\vert + \vert \Lambda^k \partial_{x} \partial_\xi \psi (\s,x,\xi,h)\vert \leq  C_k \s ,\\
& \vert \Lambda^k\partial^{2}_\xi \psi (\s,x,\xi,h) \vert \leq  C_k \tau_0 h^{-\mez}   \s,
 \end{aligned}
 \end{equation}
for all $(\s,x,\xi,h) \in  \mathcal{O}$, where  $C$ depends 
only on $ \Vert (\eta, \psi) \Vert_{L^\infty (I, H^{s+ \mez}(\xR) \times H^{s}(\xR))}$.
\end{prop}
\begin{proof}
Differentiating \eqref{eikonal} with respect to $\xi$ we find that $U= \partial_\xi \psi$ 
satisfies the equation
\begin{equation}\label{U}
\partial_\s U + (a'(\xi + h^\mez \partial_x \psi) + h^\mez W_h^\delta) \partial_x U 
= - (\partial_x \psi) \int_0^1 a''(\xi + h^\mez \lambda \partial_x \psi)dÊ\lambda  - W_h^\delta. 
\end{equation}
 Then the estimate of the first term in the first line of 
\eqref{psixi1} follows from  Corollary \ref{LU=F}.  To estimate the second term we differentiate with 
respect to $\xi$ the equation \eqref{psi1}. We find that the fuction 
$U_1= \partial_\xi \psi_1 =  \partial_x \partial_\xi \psi$ satisfies the equation  
\begin{equation*}
\partial_\s U_1 + (a'(\xi + h^\mez \psi_1) + h^\mez W_h^\delta) \partial_x U_1 
+ h^\mez a''(\xi + h^\mez  \psi_1) \partial_x \psi_1 U_1 = -\partial_x W_h^\delta. 
 \end{equation*}
 The second  estimate in the first line of \eqref{psixi1} follows from  Corollary \ref{LU=F}.  Finally to estimate 
 $U_2= \partial_\xi^2 \psi$ we differentiate \eqref{U} 
 with respect to $\xi$ and we find that $U_2$ satisfies the equation
\begin{equation*}
\partial_\s U_2 + (a'(\xi + h^\mez \partial_x \psi) + h^\mez W_h^\delta) \partial_x U_2 = F
\end{equation*}
where 
\begin{multline*}
F = -h^\mez (\partial_x \partial_\xi \psi)^2 a''(\xi + h^\mez \partial_x \psi) 
+ (\partial_x \partial_\xi \psi)a''(\xi + h^\mez \partial_x \psi) \\
+ \frac{a''(\xi + h^\mez \partial_x \psi) - a''(\xi)}{h^\mez}.
\end{multline*}
So using the estimate on ${\partial_x \psi}$ and $\partial_ x \partial_\xi \psi$ 
obtained before, we deduce from Corollary \ref{LU=F} the last estimate of \eqref{psixi1}.
\end{proof}

\subsubsection{The symbol equations}
According to the formulas \eqref{C6}--\eqref{Anouveau}, 
since the phase $\psi$ now satisfies the eikonal equation \eqref{eikonal}, 
we are lead to solve the following transport equation
\begin{equation}\label{c0}
\begin{aligned}
\partial_\s b + a'(\xi+h^\mez \partial_x\psi)\partial_x b +  h^\mez W_h^\delta \partial_x b 
&-h^\mez (\partial_x^2\psi) a''(\xi+h^\mez \partial_x\psi) b \\
&+\frac{1}{2} h^\mez(\partial_x W^\delta_h)b
=-\frac{i}{h}\sum_{k=2}^N A_k,
\end{aligned}
\end{equation}
with
\begin{equation}\label{initiale}
b\arrowvert_{\s=0}=\chi_1(\xi),
\end{equation}
where $\chi_1 \in C_0^\infty(\xR)$ is equal to one on the support of the function $\chi$ given in  Theorem  \ref{DISP2}.
Let 
$$
\mu_0=\mez-2\delta,
$$
where we recall that $\delta<1/4$.

Let us set 
\begin{equation}\label{c1}
\Lambda =h^\delta \partial_x.
\end{equation}
Then according to \eqref{Anouveau} and the Leibniz formula one can write
\begin{equation}\label{c2}
\frac{1}{h}A_k =h^{k(1-\delta)-1}\sum_{k_1=0}^k c_{k,k_1} \Ly^{k-k_1}
\left[ (\partial_\xi^k a)(\rho((x,y))\right]\big\arrowvert_{y=x} \Lambda^{k_1}b,
\end{equation}
where $c_{k,k_1}\in \xC$. 

We shall take $b$ of the form
\begin{equation}\label{c3}
b=\chi_1 (\xi)e^{\widetilde{\theta}},\quad \widetilde{\theta}
=\sum_{j=0}^M {\widetilde{\theta}_j},\quad {\widetilde{\theta}}_j =h^{j\mu_0}\theta_j.
\end{equation}
Then $\Lambda^{k_1}b$ is a finite linear combination of terms of the form
$$
\left( \Lambda^{\alpha_0}e^{\widetilde{\theta}_0} \right)\cdots \left( \Lambda^{\alpha_M}e^{\widetilde{\theta}_M} \right),
\quad \alpha_0+\cdots +\alpha_M=k_1.
$$
Let 
$$
\omega=\left\{ \alpha\in \xN^{M+1} \,:\, \la\alpha\ra=k_1\right\},
$$
and, for $0\le p\le M$,
$$
\omega_p=\left\{ \alpha\in \omega,\quad
\alpha=(\alpha_0,\cdots,\alpha_p,0,\cdots,0) \text{ with }\alpha_p\neq 0\right\}.
$$
Then $\omega$ is the disjoint union of the $\omega_p$. 
It follows that 
\begin{equation}\label{c4}
\Lambda^{k_1} b =\sum_{p=0}^{M}
\sum_{\alpha\in\omega_p} d_{p,\alpha}\left( \Lambda^{\alpha_0}e^{\widetilde{\theta}_0} \right)\cdots \left( \Lambda^{\alpha_p}e^{\widetilde{\theta}_p} \right)
\exp\left(\sum_{j=p+1}^M {\widetilde{\theta}_j}\right),\quad d_{\alpha,p}\in\xC.
\end{equation}
Now by the Faa-di Bruno formula we have for $0\le \ell \le M$,
\begin{equation}\label{c5}
\Lambda^{\alpha_\ell} e^{\widetilde{\theta}_{\ell}} =e^{\widetilde{\theta}_{\ell}}  \sum_{s=1}^{\alpha_\ell}E_{s\ell}
\end{equation}
where $E_{s,\ell}$ is a finite linear combintation of terms of the form
$$
\prod_{i=1}^{s}\left(\Lambda^{p_i} \widetilde{\theta}_\ell\right)^{q_i} \quad\text{where }
1\le \sum_{i=1}^{s}q_i\le \alpha_\ell,\quad
\sum_{i=1}^{s}p_i q_i=\alpha_\ell.
$$
Since $\sum_{i=1}^{s}q_i\ge 1$ and $p(p+1)/2\ge p$, it follows from \eqref{c4} and \eqref{c3} that
\begin{equation}\label{c6}
\Lambda^{k_1} b = e^{\widetilde{\theta}} \sum_{p=0}^M h^{p\mu_0}\sum_{\la\beta\ra\le k_1}
G_{p\beta}\left( h,\Lambda^{\beta_0}\theta_0,\ldots,\Lambda^{\beta_p}\theta_p\right),
\end{equation}
where $G_{p,\beta}(h,\zeta_0,\ldots,\zeta_p)$ are bounded in $h$ and polynomial in $\zeta$. 
Coming back to 
\eqref{c2} we remark first that since $k\ge 2$ and $\mu_0=\mez- 2\delta>0$ we have
\begin{equation}\label{c7}
k(1-\delta)-1\ge 2(1-\delta)-1=\mez+\mu_0.
\end{equation}
Let us note that this is the only point where we use the fact that $\delta < \frac{1}{4}.$

On the other hand we have,
\begin{equation}\label{c8}
\Ly^{k-k_1}\left[ (\partial_\xi^k a)(\rho((x,y))\right]\big\arrowvert_{y=x} \in S^0_\delta.
\end{equation}
It follows then from \eqref{c2}, \eqref{c6}, \eqref{c7}, \eqref{c8} that for $k\ge 2$
\begin{equation}\label{c9}
\frac{1}{h}A_k =h^\mez \sum_{p=0}^M\sum_{k_1=0}^k \sum_{\la\beta\ra\le k_1}
h^{(p+1)\mu_0} f_{k,k_1} H_{p,\beta}\left( h,\Lambda^{\beta_0}\theta_0,\ldots,\Lambda^{\beta_p}\theta_p\right)
 e^{\widetilde{\theta}},
\end{equation}
where $f_{k,k_1}\in S^{0}_\delta$, 
$H_{p,\beta}(h,\zeta_0,\ldots,\zeta_p)$ are bounded in $h$ and polynomial in $\zeta$.\\
Let us denote by $L$ the linear operator appearing in  \eqref{c0},
\begin{equation}\label{c}
L=\partial_\s  + a'(\xi+h^\mez \partial_x\psi)\partial_x +  h^\mez W_h^\delta \partial_x .
\end{equation}
Since $b=e^{\widetilde{\theta}}$ with 
$\widetilde{\theta}=\sum_{p=0}^M h^{p\mu_0} \theta_p$ we have $Lb= e^{\widetilde{\theta}}  L\widetilde{\theta} $. 
It follows from 
\eqref{c9} that the transport equation \eqref{c0} can be written, modulo a remainder, 
\begin{align*}
L\theta_0&-h^\mez \big \{ (\partial_x^2\psi) -\frac{1}{2}(\partial_x W^\delta_h) \big \} \\
&+\sum_{p=0}^{M-1}h^{(p+1)\mu_0} \left(L \theta_{p+1}-h^{\mez} G_{p}(\theta_0,\ldots,\theta_p)\right)
=0,
\end{align*}
where $G_j(\theta_0,\ldots,\theta_p)$ are polynomials in $\Lambda^\beta \theta_\ell$ for $\la\beta\ra\le N$. 
Therefore we shall take for $\theta_p, 0 \leq p \leq M-1,$ the solutions of the problems 
\begin{equation}\label{c111}
\left\{
\begin{aligned}
&L\theta_0 = h^\mez \big \{(\partial_x^2\psi) a''(\xi+h^\mez \partial_x\psi) -\frac{1}{2} (\partial_x W^\delta_h)\big  \},\quad \theta_0\arrowvert_{\s=0}=0,\\  
&L\theta_{p+1}=h^\mez G_p(\theta_0,\ldots,\theta_p), \quad \theta_{p+1}\arrowvert_{\s=0}=0\quad 
(0\le p\le M-1).
\end{aligned}
\right.
\end{equation}
 We have the following result.
\begin{prop}\label{b}
Let $\Lambda = h^\mez \partial_x$. Then  \eqref{c111} has a unique solution $(\theta_1,..., \theta_{M})$ such that for  $0 \leq p \leq M$ and all integers $k\in \xN$,
\begin{equation}\label{estb}
 \begin{aligned}
&\vert \Lambda^k \theta_p (\s,x,\xi,h) \vert \leq C_k, \quad \vert \Lambda^k \partial_\xi \theta_p (\s,x,\xi,h) \vert \leq C_k h^{ - \delta} \s,\\
&\vert \Lambda^k \partial_\xi ^2 \theta_p (\s,x,\xi,h) \vert \leq C_k h^{ -\mez-2\delta} \s.
\end{aligned}
\end{equation}
\end{prop}
\begin{proof}
We proceed by induction on $p$. If $p= 0$, the estimate of the first term in the first line of \eqref{estb} follows immediately from Proposition \ref{eikS_0} and Corollary \ref{LU=F}. Now $ \partial_\xi \theta_0$ is solution of the  equation
\begin{equation}\label{TETA}
\begin{aligned}
 L\partial_\xi \theta_0 &= h^\mez \big \{ (\partial_\xi \partial_x^2 \psi)a''(\xi+h^\mez \partial_x\psi) + (\partial_x^2 \psi)(1+ h^\mez \partial_x \partial_\xi \psi)a'''(\xi+h^\mez \partial_x\psi)\big \}\\
 &- (1+h^\mez \partial_\xi \partial_x\psi)a''(\xi+h^\mez \partial_x\psi)\partial_x \theta_0 := F_0.
 \end{aligned}
 \end{equation}
It follows from \eqref{psixi1} and the first estimate that $ \Vert \Lambda^j F_0 \Vert_{L^\infty (\mathcal{O})} \leq C_jh^{-\delta}$. Using Corollary \ref {LU=F} we obtain the estimate of the second term in the first line of \eqref{estb}. To estimate $\partial_\xi ^2 \theta_0$ we differentiate one more time the equation \eqref{TETA} and we find using the same arguments that $L\partial_\xi ^2 \theta_0 = F_1$ where $F_1$ satisfies, $ \Vert \Lambda^j F_1 \Vert_{L^\infty (\mathcal{O})} \leq C_jh^{-\mez-2\delta}$. The estimate of the term in the second line of \eqref{estb} follows the from Corollary \ref{LU=F}. Assuming that \eqref{estb} is true up to the order $p$ the estimate $\theta_{p+1}$  follows from the second equation in \eqref{c111} and the induction. 
\end{proof}
 It follows then from \eqref{r=bis},\eqref{eikonal}, \eqref{c0} and \eqref{c111} that 
\begin{equation} \label{rfinal}
r = \sum_{j=1}^5 r_j,
\end{equation}
where $r_1, r_2,r_3,r_4$ are defined in \eqref{C9},\eqref{C10},\eqref{Anouveau},\eqref{r4bis} and
$$r_5= h^\mez h^{(M+1)\mu_0}G_M(\theta_0,.., \theta_M)\chi_1(\xi)e^{\widetilde{\theta}} \zeta.$$
\end{proof}
\subsubsection{End of the proof of Proposition \ref{eikampl2}}
Since we have
$$\widetilde{b} = \chi_1(\xi)c(\s,x,\xi,h)(\zeta(x-z-\s a'(\xi))$$
with $c = e^{\widetilde{\theta}}= e^{\sum_{p=0}^M h^{p\mu_0} \theta_p} \in S^0_\delta(\Omega)$ the same arguments as those used in section 4.1.1 give the proof of \eqref{estReste2}.
\subsection{The stationary phase lemma}

In the sequel we will use the following elementary version of the stationary phase inequality where 
we allow complex valued phase functions.
\begin{lemm}\label{phasestatio}
For any real numbers $\alpha,\beta,h,\rho$ such that
$$
\alpha<\beta, \quad 0<h\leq 1, \quad \rho>0
$$
and for any functions $\phi\in C^2([\alpha,\beta],  \xC)$, 
$p\in C^1([\alpha,\beta], \xC)$ such that 
$$
\forall \xi \in [\alpha, \beta],\quad \la\IM \phi(\xi)\ra \leq h,\quad \la\IM \phi''(\xi) \ra\leq \rho,
\quad  \frac{\rho}{2} \leq \RE\left(\phi''(\xi)\right) \leq \rho, 
$$ 
we have 
$$ \left| \int_\alpha^\beta e^{\frac{i}{h}\phi(\xi)} p(\xi) \,d\xi \right| \leq 
\left( 8 \|p\|_{L^\infty}+ 2\int_\alpha^\beta \la p' (\xi)\ra \, d\xi\right)\Bigl(\frac h {\rho}\Bigr)^{\frac{1}{2}}.
$$
\end{lemm}
\begin{proof}
Notice that we can assume that $\rho \geq h$ (otherwise, the conclusion is straightforward). 
Notice now that, using the monotonicity assumption of the real part of the phase $\phi$, 
we can decompose the interval $(\alpha, \beta)$ into the disjoint union of at most three intervals 
$$
(\alpha, \beta) = I_1 \cup I_2 \cup I_3,
$$
where $I_1$, $ I_2$ or $I_3$ are possibly empty and satisfy
\begin{alignat*}{2}
&\forall \xi\in I_1, \quad&&\RE (\phi'(\xi))\leq -(\rho h) ^{1/2}, \\
&\forall \xi\in I_2, && - (\rho h) ^{1/2} \leq \RE (\phi'(\xi))\leq (\rho h) ^{1/2},\\
&\forall \xi\in I_3, &&\RE (\phi'(\xi))\geq (\rho h) ^{1/2}.
\end{alignat*}
Let us first study the contribution of $I_3$. Either $I_3=\emptyset$ or $I_3$ is an 
interval contained in $[\delta, \beta]$ for some 
$\delta\in [\alpha,\beta[$. 
Then, using that
$$
\frac{h}{i \phi'(\xi)} \partial _\xi \bigl( e^{\frac{i}{h}\phi(\xi)}\bigr) = e^{\frac{i}{h}\phi(\xi)},
$$ and integrating by parts, we obtain
 \begin{multline}
  \int_\delta^\beta e^{\frac{i}{h} \phi(\xi)} p(\xi) d\xi=
  \Bigl[  \frac h {i \phi'(\xi)}e^{\frac i h \phi(\xi)} p(\xi)\Bigr]_\delta^\beta
 - \int_\delta^\beta e^{\frac{i \phi(\xi)} h} \partial_\xi \Bigl(  \frac h {i \phi'(\xi)} p(\xi)\Bigr) \, d\xi\\
 = \Bigl[  \frac h {i \phi'(\xi)}e^{\frac i h \phi(\xi)} 
p(\xi)\Bigr]_\delta^\beta- \int_\delta^\beta e^{\frac{i \phi(\xi)} h} \frac h {i \phi'(\xi)}
\partial_\xi p(\xi)d\xi + \int_\delta^\beta e^{\frac{i \phi(\xi)} h} \frac{ h \phi''(\xi) } {\phi'(\xi)^2} p(\xi) d\xi
 \end{multline}
Clearly, the contributions of the two first terms are easily handled by means of 
the lower bound on $\RE (\phi'(\xi))$ on $I_3$, and to conclude, 
it suffices to bound the last term. But according to the assumption on the phase, 
we now have
$$
\RE (\phi'(\delta))\geq (\rho h) ^{1/2} \Rightarrow
\RE (\phi'(\xi)) \geq (\rho h)^{1/2}+ \frac {\rho} 2  (\xi- \delta),
$$
and consequently, the last term is bounded by
$$
\lA p\rA_{L^\infty}\int_\delta ^\beta \frac {4h\rho} { (2(h\rho)^{1/2} + \rho(\xi- \delta))^2}\, d\xi \leq 
\frac  {2 h\|p\|_{L^\infty}} { (h\rho)^{1/2}},
$$
where the last inequality is obtained by a straightforward computation. 

Now, of course, the contribution of $I_1$ is dealt with similarly, 
and consequently we can  focus on the contribution of $I_2$. 
Now remark that according to the assumptions on $\phi''$, the length of $I_2$ 
is smaller that $4(h/\rho)^{1/2}$, which implies 
$$
\la \int_{I_2} e^{\frac{i}{h}\phi(\xi)} p(\xi) dt \ra \leq 4\lA p \rA_{L^\infty} 
\Bigl( \frac h \rho \Bigr)^{1/2}.
$$
This completes the proof.
\end{proof}
\subsection{End of the proof of Theorem \ref{DISP2}}
As in the preceding section we have
\begin{equation}\label {S(t)=2}
\left\{
\begin{aligned}
&S(t,0,h)u_{0,h} = D_1 + D_2 + D_3 \quad where \quad  D_1 = \widetilde{U}_h (t,x),\\
 & D_2 = -S(t,0,h)v_{0,h}(x), \, D_3= - \int_{0}^t S(t,s,h) [F_h(s,x)]ds.
\end{aligned}
\right.
\end{equation}
 The terms $D_2$ and $D_3$ are estimated exactly as in section 4.3 while $D_1$ will be estimated differently using Lemma \ref{phasestatio} instead of Van der Corput Lemma. Indeed recall that according to \eqref{amplitude2} and \eqref{c3} our amplitude  in the parametrix \eqref{parametrix2} is given by
 $$\widetilde{b}(\s,x,z,\xi,h) = \chi_1(\xi)e^{\widetilde{\theta}(\s,x,z,\xi,h)}\zeta(x-z-\s a'(\xi)). $$
   
 The new fact here is that we shall glue the term $e^{\widetilde{\theta}}$ with the phase and apply the  Lemma \ref{phasestatio} with the new phase $\varphi -z\xi + \frac{h}{i} \widetilde{\theta}$. Using \eqref{parametrix2} we can write in the variable $\s = th^{-\mez}$,
 \begin{equation*}
 \begin{aligned}
 D_1&= \int K(\s,x,z,\xi,h)u_{0h}(z)\,dz ,\\
K&= \frac{1}{2\pi h} \int e^{\frac{i}{h} (\varphi(\s,x,\xi,h)-z \xi + \frac{h}{i}  \widetilde{\theta}(\s,x,\xi,h))} \chi_1(\xi) \zeta(x-z- \s a'(\xi)) \,d\xi.
\end{aligned}
\end{equation*}
Therefore we shall apply the Lemma \ref{phasestatio} with,
\begin{equation*}
\left\{
\begin{aligned}
\phi&= \varphi(\s,x,\xi,h)-z \xi + \frac{h}{i} \widetilde{\theta}(\s,x,\xi,h),\quad \varphi = x \xi + \s a(\xi) + h^\mez \psi (\s,x,\xi,h),\\
p &=  \chi_1(\xi) \zeta(x-z- \s a'(\xi)) ,\\
\rho & = C\s, \quad C>0.
\end{aligned}
\right.
\end{equation*}
Let us show that all the hypotheses in this lemma are satisfied. For this we shall use \eqref{psixi1} and \eqref{estb}. 
First of all,  since $ \delta \in ]0, \frac{1}{4}[$ we have, 
$$\vert \IM \phi \vert = h \vert \IM \widetilde {\theta} \vert \leq Ch, \qquad\vert \partial_\xi^2 \IM \phi \vert \leq  h\vert \partial_\xi^2 \widetilde {\theta} \vert \leq Ch h^{-\mez - 2\delta} \s \leq \s.$$  
Moreover,
$$ \vert \partial_\xi^2 \RE \phi \vert \leq \vert a''(\xi)\vert \s  + h^\mez \vert \partial_\xi^2 \psi \vert + h \vert \partial_\xi^2 \widetilde {\theta} \vert \leq C \s.$$
 Finally,
 $$ \vert \RE \phi \vert \geq \vert a''(\xi)\vert \s - h^\mez \vert \partial_\xi^2 \psi \vert - h \vert \partial_\xi^2 \widetilde {\theta} \vert \geq  \vert a''(\xi)\vert \s - C_1 \tau_0 \s - C_2 h^{\mez - 2 \delta} \s \geq C_3 \s, $$ if $\tau_0$ and $h$  are small enough.

It follows the from Lemma \eqref{phasestatio} that,
$$ \vert K \vert \leq Ch^{-1} \Big ( \frac h \s \Big)^\mez \Big \{1+ \int  \s \vert a''(\xi) \vert \vert \zeta' (x - z -\s a'(\xi) \vert d\xi \Big \}.$$
Since the last integral is bounded by $C' \int \vert \zeta'(t) \vert dt$, we deduce that the term $D_1$ satisfies the estimate \eqref{eq.dispersion} which completes the proof of Theorem \ref{DISP2}.

\section{Back to estimates for $(\eta, \psi)$}\label{section.6}

Notice that up to now, we only proved estimates for the dyadically localized functions $ \Delta_j u$. In this section, we shall show how we can recover estimates for $(\eta, \psi)$, the solutions of the water-wave system~\eqref{system}.   Recall that the Besov space $B^r_{\infty,2}$ is defined by 
$$
u \in B^r_{\infty,2} (\xR)\Leftrightarrow \sum_{j\in\xN}2^{2jr}\lA \Delta_j u\rA_{L^\infty}^2 <+\infty.
$$
We will use the following elementary lemma
\begin{lemm}\label{lem.triv}
 If the symbol $a \in \Gamma^m_0$, then the operator $T_a$ is bounded from $B^s_{\infty, 2}(\xR)$ to $B^{s-m}_{\infty, 2} ( \xR)$.
\end{lemm}
We have the slightly stronger result (compared to Theorems~\ref{theo:main},~\ref{theo:classic})
 \begin{prop}\label{pstrong}
Let $I = [0,T].$ Under the assumptions of Theorem~\ref{theo:main}, there exists $\eps>0$ such that
\begin{equation}\label{strong1}
(\eta,\psi) \in L^4(I,B^{s-\frac{1}{4}+\mez +\eps}_{\infty,2}(\xR)\times 
B^{s-\frac{1}{4}+\eps}_{\infty,2}(\xR)).
\end{equation}

Under the assumptions of Theorem~\ref{theo:classic}, we have
\begin{equation}\label{strong2}
(\eta,\psi) \in L^4(I,B^{s-\frac{1}{8}+\mez}_{\infty,2}(\xR)\times 
B^{s-\frac{1}{8}}_{\infty,2}(\xR)).
\end{equation}
\end{prop}
Notice that Theorem~\ref{theo:main}  follows from the first part in Proposition~\ref{pstrong}, using that, 
$$B^{\sigma+\eps}_{\infty,2}(\xR)\subset W^ {\sigma ,\infty}(\xR).
$$ On the other hand, using complex interpolation theory (see~\cite[Theorem 6.4.5 (6)]{BeLo}), we have,  with $q \geq 2$, (since $H^{\s_2} = B^{\s_2}_{2,2}$),  
$$[B^{\s_1}_{\infty,2}, H^{\s_2}]_{\frac{2}{q}} = B^{\s}_{q,2} \subset W^{\s,q},\quad \s= (1-  \frac{2}{q}) \s_1+    \frac{2}{q} \s_2.
$$ 
Taking $\s_1 = s - \frac{1}{ 8}$  and $\s_2= s$ we obtain  $\s = s - \frac {1}{8} + \frac {1}{4q}$. It follows that 
 
$$ \|\psi (t,.)\|_{W^{s - \frac{1}{8} +  \frac {1}{4q} ,q}(\xR)} \leq C \|\psi (t,.)\|_{B^{s- \frac 1 8}_{\infty, 2}(\xR)} ^{1-  \frac{2}{q}} \| \psi (t,.)\|_{H^s (\xR)}^{\frac{2}{q}} .
$$ 

 It follows that, with $(p,q)$ satisfying  $ \frac 2 p + \frac 1 q = \frac 1 2 , \quad 2 \leq q<+\infty,$ we have 
 
$$ \|\psi\|^p _{L^p(I, W^{s  - \frac{1}{8} +  \frac {1}{4q},q} (\xR))}\leq C   \|\psi\|^4_{L^4(I, B^{s- \frac 1 8}_{\infty, 2}(\xR))} \| \psi\|_{L^\infty(I, H^s(\xR))}^{p - 4}.
$$ 
which  implies Theorem~\ref{theo:classic} (the estimate for $\eta$ being similar).

Let us now turn to the proof of Proposition~\ref{pstrong}. For conciseness, we will only prove~\eqref{strong2}, the proof of~\eqref{strong1} being similar (easier). Recall that
 the function $u$ is obtained from $(\eta, \psi)$ through the following steps:
 \begin{enumerate}
 \item $u = T_{e^{ig}} \Phi^*$, where the function $g$ is real and satisfies $\partial_x g \in \Gamma^ 0_{s- \tdm}$ (which implies $e^{ig} \in \Gamma^ 0_{s- \mez}$) .
 \item $\Phi^*= \kappa^* \Phi$ where (see~\eqref{kappa_x}) $\kappa \in L^\infty(I, W^{2,\infty}(\xR)).$
 \item $\Phi = T_p \eta + i T_{c_1} U $,  with $p \in \Sigma^{1/2}_{s-1}$ is an elliptic symbol and $c_1= (1+ (\partial_x \eta)^2)^{-\frac 1 2}.$
\item $U= \psi - T_{\mathfrak{B} }\eta$, where $\mathfrak{B} \in L^\infty(I, H^{s-1}( \xR))$ is defined in~\eqref{BV}. 
\end{enumerate}

{\em Step 1:} Starting from Corollary~\ref{cor.5.2}, we have 
$$ \|u \|_{L^4(I, B^{s- \frac 1 8}_{\infty,2} )} < + \infty$$
According to the symbolic calculus, since $ e^{ig} \in \Gamma_0^0$,  we have
$$u = T_{e^{ig}} \Phi^* \Rightarrow \Phi^* = T_{e^{-ig}}u  + R_{-1} ( \Phi^*)
$$ where $R_{-1} $ is of order $-1$ (i.e. bounded from $H^s( \xR)$ to $H^{s+1} ( \xR)$) . Since $H^{s + \mez}(\xR) \subset B^s_{\infty,2}(\xR)$ it follows from Lemma~\ref{lem.triv}, the boundedness of $\Phi^*$ in $L^\infty(I, H^s( \xR))$ that 
$$ \|\Phi^* \|_{L^4(I, B^{s- \frac 1 8}_{\infty,2}(\xR) )} < + \infty.$$

{\em Step 2:} 
We have
$$ \Phi^* = \kappa^* \Phi \Leftrightarrow \Phi \circ \kappa = \Phi^* + T_{\partial_x \Phi\circ \kappa} \kappa
$$ Notice that $\partial_x \Phi\circ \kappa \in L^\infty(I \times \xR)$ and $ \partial_x \kappa\in L^\infty(I,H^{s- \mez})$. As a consequence, $ T_{\partial_x \Phi\circ \kappa} \kappa  \in L^\infty(I, H^{s+ \mez}) \subset  L^\infty(I, B^{s}_{\infty,2})$. We deduce from Step 1  
$$\|\Phi\circ \kappa \|_{L^4(I, B^{s- \frac 1 8}_{\infty,2}(\xR) )} < + \infty.$$
We conclude  that
$$\|\Phi \|_{L^4(I, B^{s- \frac 1 8}_{\infty,2} (\xR))} < + \infty,$$
by using the following lemma (with $\chi= \kappa ^{-1}$, $r= s- \frac 1 8$ and $r< \s<s$).
\begin{lemm}
Let $\s>1$. Consider  $\chi$ such that $\partial_x \chi \in W^{\s-1, \infty}(\xR)$. Then, for any $0<r< \s$,  the map $ u  \mapsto u\circ \chi $ is continuous on $B^{r}_{\infty,2}$.
\end{lemm}
 Indeed, a simple calculation shows that for any $\rho \leq \s$, the map $u \mapsto u \circ \chi$ is continuous on $W^{\rho, \infty}$, and we conclude by choosing $r_1 <r < r_2\leq \s$ with $r_i \notin \mathbb{N}$ (notice that this implies $ W^{r_i, \infty} = B^{r_i}_{\infty, \infty}$) and using the real interpolation result (see~\cite[Theorem 6.4.5 (1)]{BeLo})
$$ [B^{r_1}_{\infty, \infty}, B^{r_2} _{\infty, \infty}]_{\theta,2} = B^{r}_{\infty,2},\quad  r=  (1- \theta) r_1+ \theta r_2 .
$$

{\em Step 3:} Separating real and imaginary parts, we obtain
$$  \|T_p \eta \|_{L^4(I, B^{s- \frac 1 8}_{\infty,2}(\xR) )} + \|T_{c_1} U \|_{L^4(I, B^{s- \frac 1 8}_{\infty,2} (\xR))} < + \infty$$ and the same proof as in Step 1  (using that $p$ is elliptic,  $s- \tdm \geq 1 \Rightarrow p^{-1} \in \Gamma^{-\mez }_{1} $ and for fixed $t$,  $c_1^{-1}(t,.)  \in W^{1, \infty}(\xR)\subset \Gamma^{0}_1$) gives
$$\| \eta \|_{L^4(I, B^{s- \frac 1 8+ \mez}_{\infty,2}(\xR) )}+\| U \|_{L^4(I, B^{s- \frac 1 8}_{\infty,2}(\xR) )} < + \infty.$$

{\em Step 4:} We have $\psi = U + T_{\mathfrak{B}} \eta$. So    using the boundedness of $\eta$ in $L^\infty(I, H^{s+ \mez}(\xR))$, of  $\mathfrak{B}$ in $L^\infty(I, H^{s-1}( \xR))\subset L^\infty(I \times \xR) $ and Sobolev injections, we obtain,
$$\| \psi \|_{L^4(I, B^{s- \frac 1 8}_{\infty,2} )} < + \infty,$$
which completes the proof of Proposition~\ref{pstrong} and consequently of Theorems~\ref{theo:main} and~\ref{theo:classic}.
 
\section{Appendix} 
In this section we give a proof of \eqref{C6} to \eqref{C10}.

 Let $a \in C_0^\infty (\xR)$  with $\supp a \subset \{ \vert \xi \vert \leq C_0 \}, b \in  C_0^\infty (\xR)$ 
 and $\varphi \in C^\infty(\xR)$ real valued such that $\sup \vert \frac{\partial \varphi}{\partial x}(x) \vert \leq C_0.$ 
 Let us set
 \begin{equation}
 I = e^{- \frac{i}{h} \varphi(x)}a(hD) \Big( b\, e^{ \frac{i}{h} \varphi}  \Big)(x).
 \end{equation}
We have
$$ I = (2 \pi h)^{-1} \iint   e^{\frac{i}{h}((x - y) \xi + \varphi(y) - \varphi(x))} a(\xi) b(y) dy d \xi. $$
Moreover we can write
$$\varphi(x) - \varphi(y) = (x - y) \rho (x,y), \quad \rho (x,y) =
\int_0^1 \frac{\partial \varphi}{\partial x}(\lambda x + (1- \lambda )y) d \lambda.$$
We have  $\vert \rho \vert \leq C_0$  so, setting $ \eta = \xi - \rho(x,y) $ 
we obtain,
$$
I = (2 \pi h)^{-1} \iint e^ {\frac {i}{h}(x - y) \eta} \kappa_0(\eta) a( \eta + \rho((x,y))b(y) \,dyd\eta
$$
where $ \kappa_0 \in C_0^\infty (\xR)$ is such that 
$\chi_0(\eta)=1 $ if $ \vert \eta \vert \leq 2C_0.$ 

Using the Taylor expansion of the function $a$ 
at the point $\theta (x,y)$ we obtain $I = I_1 +R_1$ 
where,
$$ I_1 = (2 \pi h)^{-1} \sum_{k = 0}^{M-1} \frac{1}{k!}
\iint  \eta^k e^{\frac{i}{h}(x - y) \eta} \kappa_0 (\eta)  a^{(k)}(\theta (x,y)) b(y) dy d \eta
$$
and,
$$
R_1 = c_M h^{-1}\iint \eta^M  e^{\frac{i(x-y)\eta}{h}}
\kappa_0 (\eta) \int _0^1 (1 -\lambda)^{M-1} a^{(M)}(\lambda \eta + \rho( (x,y))b(y) d \lambda dy d \eta.
$$
Now we have $\eta^k e^{\frac{i}{h}(x - y) \eta}= (-\frac{h}{i} \partial_y)^ke^{\frac{i}{h}(x - y)\eta},$ so integrating 
by parts in the integrals $I_1$ and $R_1$ we obtain,
\begin{align*}
   I_1 &= (2 \pi h)^{-1} \sum_{k = 0}^{M-1} \frac{h^k}{i^k k!} \iint    e^{\frac{i}{h}(x - y) \eta} \kappa_0 (\eta)  \partial_y^k [ a^{(k)}(\theta (x,y)) b(y)] dy d \eta\\
   I_1&=(2 \pi h)^{-1} \sum_{k = 0}^{M-1} \frac{h^k}{i^k k!}  \int \hat{\kappa}_0 (\frac{x-y}{h})  \partial_y^k [ a^{(k)}(\theta (x,y)) b(y)] dy,
    \end{align*}
 $$  R_1= c' h^{M-1}\iiint_0^1  e^{\frac{i(x - y) \eta}{h}} \kappa_0 (\eta) (1 -\lambda)^{M-1} \partial_y^M[a^{(M)}(\lambda \eta + \rho( (x,y))b(y)] d \lambda dy d \eta.$$

Let us set
$$
f(x,y) = \partial_y^k [ a^{(k)}(\theta (x,y)) b(y)].
$$
Now we set  in the integral, $x-y =hz$ and we write,
\begin{multline*}
f(x,x-hz) = \sum_{j=0}^{M-1} \frac{(-hz)^j}{j!}(\partial_y^j f)(x,x) \\+ \frac{(-hz)^M}{(M-1)!} 
\int_0^1 (1- \lambda)^{M-1}(\partial_y^Mf) (x,x-\lambda hz) d \lambda.
\end{multline*}
Then we use the following equality, which reflects the fact that $\kappa_0$ is equal to one near the origin. For $j \in \xN$ we have,
$$
\int z^j \hat{\kappa}_0(z) dz = 2 \pi \delta_{j,0},
$$
where $\delta_{j,0}$ is the Kronecker symbol. 
It follows that,
$$
I_1 =  \sum_{k = 0}^{M-1} \frac{h^k}{i^k k!}  \partial_y^k [ a^{(k)}(\theta (x,y)) b(y)] \big\arrowvert_{y=x}+ R_2,
$$
where $R_2= \sum_{k = 0}^{M-1}c_{k}h^{k+M} r_k$ with
$$
r_k=\iint_0^1 z^M \hat {\kappa}_0(z) (1-\lambda)^{M-1}\partial_y^{M+k}[ a^{(M)}(\theta (x,y)) b(y)] \big\arrowvert_{y=x-\lambda hz} dÊ\lambda dz.
$$
Thus we obtain \eqref{C6} to \eqref{C10}.


\end{document}